\documentclass[11pt,a4paper,reqno]{amsart}
\usepackage[applemac]{inputenc}
\usepackage[T1]{fontenc}
\usepackage{hyperref}
\usepackage{amsmath}
\usepackage{amsthm}
\usepackage{amsfonts}
\usepackage{amssymb}
\usepackage{graphicx}
\usepackage{color}
\usepackage{amsbsy}
\usepackage{mathrsfs}
\usepackage{bbm}
\usepackage{verbatim}
\usepackage{xcolor}
\usepackage{ulem}
\usepackage{dsfont}
\usepackage{graphicx}
\usepackage{subcaption}

\newtheorem{theorem}{Theorem}[section]
\newtheorem{lemma}[theorem]{Lemma}

\newtheorem{proposition}[theorem]{Proposition}
\newtheorem{corollary}[theorem]{Corollary}
\newtheorem{definition}[theorem]{Definition}

\theoremstyle{remark}
\newtheorem{remark}[theorem]{\it \bf{Remark}\/}
\newenvironment{acknowledgement}{\noindent{\bf Acknowledgement.~}}{}
\numberwithin{equation}{section}
\catcode`@=11
\def\section{\@startsection{section}{1}%
	\z@{1.5\linespacing\@plus\linespacing}{.5\linespacing}%
	{\normalfont\bfseries\large\centering}}
\catcode`@=12
\newcommand{\be}{\begin{equation}}
	\newcommand{\ee}{\end{equation}}
\newcommand{\bea}{\begin{eqnarray}}
	\newcommand{\eea}{\end{eqnarray}}
\newcommand{\bee}{\begin{eqnarray*}}
	\newcommand{\eee}{\end{eqnarray*}}
\def\div{{\rm div \;}}

\def\pa{\partial}

\def\na{\nabla}

\def\CC{\mathbb{C}}
\def\NN{\mathbb{N}}
\def\RR{\mathbb{R}}

\def\ep{\varepsilon}

\def\calC{{\mathcal C}}
\def\calB{{\mathcal B}}

\def\calD{{\mathcal D}}

\def\calL{{\mathcal L}}

\def\calR{{\mathcal R}}

\def\calT{{\mathcal T}}

\def\calT{{\mathcal T}}

\catcode`@=11
\def\supess{\mathop{\operator@font Sup\,ess}}
\catcode`@=12

\def\div{{\rm div \;}}

\def\CC{\mathbb{C}}

\def\NN{\mathbb{N}}

\def\RR{\mathbb{R}}
\def\CC{\mathbb{C}}

\def\PP{\mathbb{P}}

\def\a{\alpha}
\def\e{\varepsilon}

\def\bar#1{{\overline #1}}

\def\R2+{\RR ^2_+}

\def\pa{\partial}
\def\na{\nabla}

\def\lim{\mathop{\rm lim}}

\def\sup{\mathop{\rm sup}}

\def\l{\lambda}

\def\log{{\rm log}}

\def\pa{\partial}

\def\pa{\partial}
\def\la{\langle}

\def\ra{\rangle}

\def\Dein{\Delta^{-1}}
\def\Ltwo{{_{L^2}}}
\def\delg{\delta_g}
\def\tpu{\tilde{P}_u}
\def\phir{\tilde \varphi_{i}}
\def\phirj{\tilde \varphi_{j}}

\def\bbf{\textbf{f}}
\def\tepu{\tilde{\ep}_u}
\def\teps{\tilde{\ep}_s}
\def\tps{\tilde{P}_s}
\begin{document}
	
	\title[]{Finite-time blowup for Keller-Segel-Navier-Stokes system in three dimensions}
	
	\author[Z. Li]{Zexing Li}
	\address{DPMMS\\Center for Mathematical Sciences\\Univeristy of Cambridge\\Wilberforce road\\
		Cambridge\\ UK}
	\email{zl486@cam.ac.uk}

 \author[T. Zhou]{Tao Zhou}
 \address{Department of Mathematics\\
Faculty of Science,
National University of Singapore\\
Singapore}
\email{zhoutao@u.nus.edu}

	\maketitle

 \begin{abstract}
     While finite-time blowup solutions have been studied in depth for the Keller-Segel equation, a fundamental model describing chemotaxis, the existence of finite-time blowup solutions to chemotaxis-fluid models remains largely unexplored. To fill this gap in the literature, we use a quantitative method to directly construct a smooth finite-time blowup solution for the Keller-Segel-Navier-Stokes system with buoyancy in $3D$. The heart of the proof is to establish the non-radial finite-codimensional stability of an explicit self-similar blowup solution to $3D$ Keller-Segel equation with the abstract semigroup tool from \cite{MR4359478}, which partially generalizes the radial stability result \cite{MR4685953} to the non-radial setting. Additionally, we introduce a robust localization argument to find blowup solutions with non-negative density and finite mass.
     
 \end{abstract}

\section{Introduction}

Chemotaxis is a ubiquitous phenomenon observed in nature, where organisms like body cells and bacteria respond directionally to the chemical substance in the external environment.
A fundamental mathematical model for chemotaxis is the Keller-Segel equation, formulated as follows:
\be
\begin{cases}
\partial_t \rho  = \Delta \rho - \na \cdot (\rho \na c), \\
-\Delta c = \rho.
\end{cases}\tag{KS}
\label{equation, Keller-Segel}
\ee
Here, $\rho$ represents the cell density, while $c$ denotes the concentration of the self-emitted chemical substance. The system describes how cells exhibit random Brownian motion, resulting in dissipation. Simultaneously, the cells are attracted by chemical substances and move to the region with higher concentrations of chemical substances. For more background on chemotaxis and related models, interested readers can refer to \cite{MR3932458, MR2013508, MR2073515}.

On the other hand, many studies consider more practical chemotaxis models that are coupled with the ambient flow, and interested readers can refer to \cite{hillesdon1995development, MR1145013, tuval2005bacterial} for further motivations regarding the introduction of fluid-chemotaxis models. These models describe that cells are not solely attracted by self-emitted chemical substances, but are also transported by the ambient flow.
In particular, this paper focuses on the $3D$ coupled Keller-Segel-Navier-Stokes system with buoyancy
\be
	\begin{cases}
		\partial_t \rho + u \cdot \nabla \rho = \Delta \rho - \nabla \cdot (\rho \nabla c), \\
		-\Delta c = \rho, \\
		\partial_t u + u \cdot \nabla u = \Delta u - \nabla \pi  - \rho e_3, \\
		\nabla \cdot u =0,
	\end{cases}  \tag{KS-NS}
	\label{equation: NS-KS velocity form}
\ee
where $e_3 = (0,0,1)$ and the vector field $u$, representing the ambient fluid flow, satisfies the Navier-Stokes equation with an additional force term $-\rho e_3$ on the right-hand side, which denotes the cell's reaction force to the buoyancy exerted by the fluid.

\subsection{Background}

\subsubsection{Chemotaxis and singularity formation.} 
\label{subsection111}

For Keller-Segel equation \eqref{equation, Keller-Segel}, the total mass $M(t) = \int \rho(t,x) dx$ is conserved. In addition, the model \eqref{equation, Keller-Segel} satisfies the following scaling invariance: if $(\rho,c)$ is the solution to \eqref{equation, Keller-Segel}, so is
\[
(\rho_\lambda,c_\lambda) = \left( \frac{1}{\lambda^2} \rho\left(\frac{t}{\lambda^2}, \frac{x}{\lambda} \right),  c\left( \frac{t}{\lambda^2}, \frac{x}{\lambda} \right) \right), \quad \l > 0.
\]

In two dimensions, the total mass $M$ plays a crucial role in determining whether a solution exhibits global existence or finite-time blowup, with $M=8\pi$ serving as the threshold distinguishing between the two possibilities. Specifically, when $M> 8 \pi$ and the initial data $\rho_0 \in L_+^1((1+|x|^2),dx)$, after tracking the evolution of the second moment $M_2[\rho(t)] := \int \rho(t,x) |x|^2 dx$, it can be confirmed that the solution would blow up in finite time (see \cite{Blanchet_Dolbeault_Perthame_globalexistence06, Dolbeault_Perthame_globalexistence04} for details). In addition, Blanchet, Dolbeault, and Perthame \cite{ Blanchet_Dolbeault_Perthame_globalexistence06, Dolbeault_Perthame_globalexistence04} established a uniform $L^\infty$ bound for the solution with mass $M<8 \pi$, proving the global existence.

In three dimensions, it is known that there is no critical mass: namely, there exist radial blowup solutions with arbitrarily small mass \cite{MR1361006}. For general solutions that are not necessarily radial, Corrias, Perthame, and Zaag \cite{MR2099126} proved blowup will happen as long as the second moment of the initial data is sufficiently small compared to its mass, whereas weak solutions exist globally if the initial data has small $L^\frac{3}{2}$ norm. Interested readers can refer to \cite{MR4201903, MR3411100, MR3438649, MR3936129} for more results.

When it comes to the singularity formation for the $2D$ Keller-Segel equation, with the model in the $L^1$ critical case, Naito and Suzuki \cite{Naito_Suzuki_typeIIblowup} verified that any finite-time blowup solution to \eqref{equation, Keller-Segel} is of type II\footnote{The solution of \eqref{equation, Keller-Segel} exhibits type I blowup at $t=T$ if 
\begin{equation*}
    \limsup_{t \to T} (T-t) \| \rho(t) \|_{L^\infty} < \infty,
\end{equation*}
otherwise, the blowup is of type II.
}. Later, in the radial setting, by using the tail analysis, Rapha\"{e}l and Schweyer \cite{Raphael_Schweyer_2DtypeII_blowup14} precisely constructed a stable finite-blowup solution of the form
\be
\rho(t,x) \simeq \frac{1}{\lambda^2(t)} U\left(\frac{x}{\lambda(t)} \right), \quad U(x) = \frac{8}{(1+|x|^2)^2}, \nonumber
\ee
with the blowup rate
\[
\lambda(t) = \sqrt{T-t} e^{-\sqrt{\frac{|\ln (T-t)|}{2}}+ O(1)}, \quad t \to T.
\]
Subsequently, Collot, Ghoul, Masmoudi, and Nguyen \cite{Collot_Ghoul_Masmoudi_Nguyen_2DtypeII_blowup22} utilized spectral theory to extend stability to the nonradial setting and obtained a more precise blowup rate. Meanwhile, they also constructed countably many unstable blowup solutions. More recently, Buseghin, Davila, Del Pino, and Musso \cite{buseghin2023existence} employed the inner-outer gluing method to further extend these findings to multiple profiles.

Compared with $2D$ case, the singularity formation for the $3D$ Keller-Segel equation \eqref{equation, Keller-Segel} is much more diverse, with various blowup profiles and formations known to exist. 
A countable family of self-similar blowup profiles have been identified by Herrero, Medina, Vel\'azquez \cite{MR1651769} and Brenner, Constantin, Kadanoff, Schenkel, Venkataramani \cite{Brenner_Constantin_Leo_Schenkel_Venkataramani_steady_state_99}. 
Moreover, in \cite{Brenner_Constantin_Leo_Schenkel_Venkataramani_steady_state_99}, they numerically verified the stability of an explicit self-similar solution given by
\be
\rho(t,x) = \frac{1}{T-t}Q\left( \frac{x}{\sqrt{T-t}} \right), \quad \text{ with } \quad Q(x) = \frac{4(6+|x|^2)}{(2 + |x|^2)^2}.
\label{Q: definition}
\tag{Profile}
\ee
 Later,  {Glogi\'c and Sch\"orkhuber \cite{MR4685953}} rigorously established its stability in the radial case using semigroup theory. To the best of the authors' knowledge, it is still open for stability when extended to nonradial case. In addition, other non-self-similar blowup solutions to the $3D$ Keller-Segel equation have been identified. For example, Collot, Ghoul, Masmoudi, and Nguyen \cite{Collot_Ghoul_Masmoudi_Nguyen_3Dblowup_Collasping-ring_blowup23} discovered a type II blowup solution that concentrates in a thin layer outside the origin and collapses towards the center. More recently, Nguyen, Nouaili, and Zaag \cite{nguyen2023construction} found a type I-Log blowup solution for the $3D$ Keller-Segel equation.

\subsubsection{Suppression of blowup by fluids}

\mbox{}

When considering the chemotaxis-fluid models such as \eqref{equation: NS-KS velocity form}, the dynamics become more intricate. Upon cell aggregation, the buoyancy/friction-induced enhancement of the ambient flow might suppress the cell aggregation process. The literature on fluid-chemotaxis models is extensive, particularly regarding the suppression of system blowup. 

Notably, significant progress has been made in the field of the suppression of chemotactic blowup by the ambient fluid flow, building upon the pioneering work of Kiselev and Xu \cite{Kiselevandxusuppresionof}, where they identified the existence of passive flow\footnote{Here, "passive" means that the flow is independent of the cells, otherwise it is called active flow.}, capable of suppressing finite-time blowup in both $2D$ and $3D$. Subsequently, Bedrossian and He \cite{MR3730537,MR3892424,MR3826109} confirmed blowup suppression via strong shear flow in both $2D$ and $3D$. 

Regarding active flow related to the Keller-Segel-Navier-Stokes system, Zeng, Zhang and Zi \cite{MR4222377} and He \cite{MR4612142} demonstrated blowup suppression near large Couette flow in $2D$, while Li, Xiang and Xu \cite{li2023suppression} revealed blowup suppression near strong Poiseuille flow in $2D$. Additionally, Cui, Wang, and Wang \cite{cui2024suppression} identified blowup suppression near strong non-parallel shear flow in $3D$. These results all fall within the regime where the cell dynamics are perturbative with respect to the flow, so as to exploit the enhanced dissipation effect induced by fluid-mixing.

Beyond the perturbative regimes mentioned previously, it is natural to consider large data dynamics. Recently, Hu, Kiselev and Yao \cite{hukiselevyao2023suppression} found a mechanism of suppression by buoyancy. This discovery led to the verification of global existence for arbitrary smooth data concerning the Keller-Segel equation coupled with a fluid flow adhering to Darcy's law for incompressible porous media via buoyancy force within a domain with a lower boundary within two dimensions. Furthermore, Hu and Kiselev \cite{kevinSashastokesboussi} confirmed the suppression of blowup through sufficiently strong buoyancy within the context of Stokes-Boussinesq flow with a cold boundary both in $2D$ and $3D$. In addition, He and Gong \cite{He_Gong_8pithresholdforNSKS} showed the global existence of solutions to the $2D$ coupled Keller-Segel-Navier-Stokes system with friction force when the total mass $M< 8 \pi$.

Despite the abundant literature on the suppression of blowup, to the best of the authors' knowledge, there is limited literature to study whether the finite-time blowup would happen for fluid-chemotaxis models. Precisely speaking, in a bounded planar domain, when taking the advection of an arbitrary passive flow into account,  Winkler \cite{winkler_blowup_passive_flow} proved the existence of finite-time blowup as long as the cell is sufficiently concentrated initially. However, when it takes the active flow into account, only numerical evidence \cite{MR2901320} suggests that cell aggregation (i.e. blowup) is likely to happen.

\subsection{Main result}

\mbox{}

Our objective in this paper is to fill the gap in the literature regarding the existence of finite-time blowup solutions for anisotropic chemotaxis-fluid models. Specifically, our main result establishes the existence of a smooth finite-time blowup solution to \eqref{equation: NS-KS velocity form} with nonnegative density and finite mass, stated as follows.

\begin{theorem}[Existence of smooth finite-time blowup solution with nonnegative density and finite mass]
\label{main thm}
For any integer $s \ge 3$ and  any divergence-free vector field $u_0 \in H_\sigma^\infty(\RR^3)$ fixed, there exists non-negative $\rho_0 \in C^\infty_0(\mathbb{R}^3)$, such that the smooth solution to \eqref{equation: NS-KS velocity form} with initial data $(\rho_0, u_0)$ blows up at some time $t=T<\infty$. Moreover, we have
\be
	\rho(t,x) =\frac{1}{T-t} \left[ Q\left( \frac{x}{\sqrt{T-t}} \right) + \ep \left(t, \frac{x}{\sqrt{T-t}} \right) \right], \quad \text{ for } x\in\mathbb{R}^3, \;  t\in [0,T),
 \label{blowup formation}
 \ee
 where $Q$ is given by \eqref{Q: definition} and $\lim_{t\to T^-}\| \ep(t) \|_{H^s(\RR^3)} = 0$.
\end{theorem}

\mbox{}

\noindent \textit{Comments to the main result.}

\mbox{}

\noindent\textit{1. Direct construction of blowup solution.} 
\mbox{}

In the literature, most blowup results for chemotaxis equations are obtained either by tracking the evolution of some appropriate functional (such as the second moment) to obtain a contradiction  \cite{Blanchet_Dolbeault_Perthame_globalexistence06, MR2099126, Dolbeault_Perthame_globalexistence04}, or by working in the radial setting where the mass accumulation function satisfies an ODE  \cite{Fuest_blowup_optimal_2021}. However, the coupling with fluid equations destroys the structures that these proofs rely on, making them inapplicable.

Therefore, we turn to the strategy of direct construction via stability analysis of an approximate blowup solution, often involving modulation or dynamical rescaling method. This strategy has found great success for parabolic, dispersive and fluid problems. In addition to works on Keller-Segel equation mentioned in Subsection \ref{subsection111}, other instances include nonlinear heat equation \cite{MR3986939,MR1427848},  nonlinear wave equation \cite{MR3537340, kim2022self}, nonliear Schr\"odinger equation \cite{MR4359478, MR2729284},  incompressible Euler equation \cite{Chen_Hou_Euler_blowup_theory_2023, Chen_Hou_Euler_blowup_numerical_2023,MR4334974, MR4445341}, compressible fluids \cite{buckmaster2022smooth, MR4445443} and so forth.

\mbox{}

{  
\noindent\textit{2. Blowup behavior of the fluid.}

In our setting, the fluid not only is unable to suppress the chemotactic blowup, but on the contrary, the self-similar blowup of chemotaxis \eqref{blowup formation} also generates a singularity in the fluid at the blowup time $T$. Precisely, we will verify the following estimates in Appendix \ref{appB}: there exists $C_1, C_2 > 1$ such that 
\begin{align}
        & C_1^{-1} |\log (T-t)| \le \| u(t) \|_{L^\infty} \le C_1 |\log (T-t)|, \quad {\rm as}\,\, t\to T;
        \label{growth of u(t)}
        \\
        & C_2^{-1} (T-t)^{-1} \le \| \na \pi (t) \|_{L^\infty} \le C_2 (T-t)^{-1},   \qquad {\rm as}\,\, t\to T.
        \label{growth of body pressure nabla pi}
    \end{align}
    In particular, compared to the blowup rate of the density (see \eqref{blowup formation}), the fluid blows up at a much slower rate near the blowup time $T$. Furthermore, from \eqref{growth of u(t)}, note that the critical space-time norm $\| u \|_{L^2_t([0, T], L^\infty_x)} < \infty$ remains finite, which guarantees non-existence of blowup on $[0,T]$ for the Navier-Stokes equations if there is no buoyancy force term $-\rho e_3$ on the right-hand side (see \cite[Theorem $4.9$]{MR4475666}). Hence the blowup in the Navier-Stokes flow \eqref{growth of u(t)} is essentially driven by the blowup in the chemotaxis, instead of the intrinsic element of the Navier-Stokes flow.
 }

\mbox{}

\noindent\textit{3. Strategy for control of the self-similar flow in nonradial setting.} 
\mbox{}

We choose the self-similar solution $Q$ to \eqref{equation, Keller-Segel} given in \eqref{Q: definition} as an approximate solution to \eqref{equation: NS-KS velocity form} and study the evolution of perturbation near $Q$ in the self-similar coordinate. The fluid part is treated perturbatively, based on the crucial observation that the Navier-Stokes equation is subcritical in scaling with respect to the Keller-Segel equation.\footnote{This is different from the model in \cite{hukiselevyao2023suppression} where the fluid part is supercritical.} The main task is to obtain decay of linearized flow near $Q$ for the Keller-Segel part.

In the radial setting,  {Glogi\'c and Sch\"orkhuber} \cite{MR4685953} have exploited the reduced-mass formulation of \eqref{equation, Keller-Segel} and the self-adjointness of the linearized operator in a certain weighted space to determine its spectrum and obtain semigroup decay. However, due to the anisotropic nature of buoyancy in \eqref{equation: NS-KS velocity form}, we have to extend to the nonradial setting. { Based on the abstract semigroup theory constructed by  Merle-Rapha\"{e}l-Rodnianski-Szeftel \cite{MR4359478} (which was further built upon Engel-Nagel \cite{engel2000one})}, the main idea is to show the linearized operator as the compact perturbation of a maximal accretive operator, which yields the semigroup decay modulo finite unstable directions (See Proposition \ref{proposition: spectral properties of L}). {  Apart from direct applications of their method \cite{buckmaster2022smooth,kim2022self,MR4445443}, this idea has found extensive application earlier. For example, Jia and \v{S}ver\'ak \cite{jiaSverakillposedness} employed it to study forward self-similar solutions of the Navier-Stokes equation, while Donninger and Sch\"orkhuber \cite{MR2909934,MR3537340} applied the similar idea to investigate self-similar blowup in wave equation.}
Another framework for semigroup decay (and spacetime estimates) without quantitative spectral analysis is presented in \cite{li2023stability}.

Finally, we integrate this semigroup estimate with a higher-order energy estimate to overcome the loss of derivative in the nonlinearity. Subsequently, we use Brouwer's argument to select initial data for decay in unstable directions.

\mbox{}

\noindent\textit{4. On the stability of self-similar profile.} 
\mbox{}

Following the standard approach in \cite[Proposition $4.10$]{MR3986939}, we can construct a finite-codimensional Lipschitz manifold of nonradial initial data such that the corresponding solution to system (\ref{equation: NS-KS velocity form}) blows up in finite time.  After dropping the perturbative estimates for fluid part, the proof also works for \eqref{equation, Keller-Segel}, implying the finite-codimensional asymptotic stability of $Q$ in nonradial setting. To consider the nonradial stability problem, it requires to count the precise number of unstable directions of the linearized operator in the nonradial setting, which remains to be addressed. In particular, the radial case was done in \cite{MR4685953}.

\mbox{}

\noindent\textit{5. Non-negativity of the density and finiteness of the mass.} 

In real-life scenarios, the density of the cell is always non-negative, and the total mass of the cell is finite. However, as indicated by \eqref{Q: definition}, $Q \notin L^p(\mathbb{R}^3)$ for any $p \in \left[1, \frac{3}{2} \right]$, implying that the approximate solution has infinite mass. 

To obtain finite-mass/energy blowup solution, one natural strategy is to truncate the profile far away and show it induces the analogous self-similar blowup behavior, which can be adapted for stable profile \cite{MR4445341, MR4685953}. A direct generalization for finite-codimensional stable blowup is when the instabilities are either generated by symmetry or of finite-mass \cite{MR3986939}. However, it requires delicate spectral analysis to completely characterize the unstable spectrum. In addition, \cite{MR4359478} (also see \cite{kim2022self}) proposed another method to address the general finite-codimensional stability case by working in a weighted Sobolev space adapted to time-dependent dampened profile and modifying the choice of initial data suitably in the Brouwer argument. 

In this paper, we propose a different approach to constructing a compactly supported smooth initial data based on the finite-codimensional stability of the profile. We reparametrize the spectral decomposition to localize the unstable eigenmodes instead of the profile,  which leads to both the compact support and non-negativity of the initial density. This modified spectral decomposition retains the diagonal part of the linear evolution, so it does not obstruct the bootstrap estimate (see Subsection \ref{subsection: modified spectral of L} for details). Despite its apparent simplicity, this method has effectively addressed our problem and is robust enough to apply to other various models.

\mbox{}

{ \noindent \noindent\textit{6. Extension of Theorem \ref{main thm}.} 

Our result can extend to more general subcritical buoyancy force $\vec\Phi(\rho)$ in replacement of $-\rho e_3$. For instance, one can take $\vec\Phi(\rho) = -\rho^\alpha e_3$ with $\alpha < \frac{3}{2}$ (such force was discussed in  \cite[(NCS)]{MR3991564} for instance). Besides, the argument presented in this paper do not depend on the specific form of the self-similar profile $Q$, so that the construction and stability analysis can accommodate other radial self-similar profiles of \eqref{equation, Keller-Segel} constructed in \cite{Brenner_Constantin_Leo_Schenkel_Venkataramani_steady_state_99,MR1651769}. 
}

\subsection{Structure of the paper}

\mbox{}

 In Section \ref{section: linear theory}, we firstly introduce the self-similar variable and formulate the singularity formation as the perturbation of the self-similar profile $Q$, then consider the spectral properties of the linearized operator $\mathcal{L}$ and the semigroup decay modulo finite (modified) unstable directions. Section \ref{section: bootstrap} is devoted to the nonlinear dynamics to complete the proof of Theorem \ref{main thm}. We leave the standard local well-posedness theory for \eqref{equation: NS-KS velocity form} in Appendix \ref{appendix: LWP}.

\subsection{Notations}

\mbox{}

Firstly, for any $R >0$, we define $B(0,R)$ as the ball centered at the origin with radial $R$ in three dimensions. Then, we choose $0 \le \chi \le 1 $ as a smooth cut-off function in $B(0,R)$ defined by
\be
\chi(x) :=
\begin{cases}
	1, & |x| \le 1, \\
	0, & |x| \ge 2,
\end{cases}
\qquad 
\chi_R(x) := \chi\left( \frac{x}{R} \right).
\label{cutoff function: defi}
\ee

Next, we call the vector of form $\alpha = (\alpha_1, \alpha_2, \alpha_3) \in \NN^3$ the multi-index of order $|\alpha|=\alpha_1 + \alpha_2 +\alpha_3$. For any given scalar function $f,g$, we define the partial derivative of $f$ with respect to the multi-index $\alpha$ by
\[
\partial^\alpha f(x) = \partial_{x_1}^{\alpha_1} \partial_{x_2}^{\alpha_2} \partial_{x_3}^{\alpha_3} f(x),
\]
and for any $k \ge 0$, we denote $D^k$ by
\[
D^k f = (\partial^\alpha f)_{|\alpha|=k}.
\]
In particular, for $k=1$, we simplify the notation into $Df =(\partial_{x_1} f, \partial_{x_2}f , \partial_{x_3} f)$, which is the gradient of $f$. Hence for any integer $k \ge 0$, we define the inner product on $\dot H^k$ and $H^k$ by 
\[ (f,g)_{\dot H^k} = (D^k f, D^k g)_{L^2} := \sum_{|\alpha|=k}(\partial^\alpha f, \partial^\alpha g)_{L^2},\quad (f,g)_{H^k} = (f,g)_{L^2} +  (f,g)_{\dot H^k}.\]
We also write $H^\infty = \cap_{k \ge 0} H^k$. 
Next, we denote $H^{k} (\RR^3; \RR)$, $H^k(\RR^3;\CC)$ (or $\dot H^{k} (\RR^3; \RR)$, $\dot H^k (\RR^3; \CC$) for the collection of all $\RR$-valued or $\CC$-valued functions respectively with finite $H^k$ (or $\dot H^k$) norm.
Out of the simplicity of notation, we will use $H^k$ (or $\dot H^k$) to refer to $H^k(\RR^3; \CC)$ (or $\dot H^k(\RR^3; \CC)$) in Subsection \ref{subsection: spectral property of L0}-\ref{subsection: accretivity and maximality of L}, and to refer to $H^k(\RR^3; \RR)$ (or $\dot H^k(\RR^3; \RR)$) in Subsection \ref{subsection: modified spectral of L} and Section \ref{section: bootstrap}. 

Similarly, for any given vector-valued function $u =(u_1,u_2,u_3)$, we can also define $\pa^\alpha u$ by
$
\partial^\alpha u = (\partial^\alpha u_1, \partial^\alpha u_2, \partial^\alpha u_3),
$
 $D^k u$ by
$
D^k u = (\partial^\alpha u)_{|\alpha|=k},
$
and thereafter the vector-valued $H^k(\RR^3;\RR^3)$ (or $\dot{H}^k(\RR^3;\RR^3)$) space similarly. In addition, we denote $H^k_{\sigma}(\RR^3)$ the set of all divergence-free vector fields within the $H^k(\RR^3;\RR^3)$ space, namely 
\[
  H^k_\sigma(\RR^3) := \{ u \in H^k(\RR^3; \RR^3): \nabla \cdot u = 0 \}.
\]

Moreover, we denote $C_0^\infty$ for the set of infinitely differentiable functions with compact support, and similarly  
$\mathscr{S}$ for Schwartz functions. 

Finally, if $A$ is a linear operator on a Hilbert space $H$, then we denote $\rho(A)$, $\sigma(A)$ to be the resolvent set and spectral set of $A$ respectively.

\mbox{}\\

\begin{acknowledgement}
Z.L. is partially supported by the ERC advanced grant SWAT. T.Z. is partially supported by MOE Tier $1$ grant A-0008491-00-00. The authors gratefully thank Yao Yao for valuable discussions and insightful suggestions during the course of this research. We sincerely thank the reviewers for their careful reading and informative suggestions. Special thanks are also extended to Charles Collot, Mahir Had\v{z}i\'c, Zhongtian Hu, and Van Tien Nguyen for their helpful discussions.
\end{acknowledgement}

\section{Linear theory in the self-similar coordinate}
\label{section: linear theory}
\subsection{Self-similar coordinate and renormalization}
Firstly, we renormalize the system \eqref{equation: NS-KS velocity form} in the self-similar coordinate. We define
\be
y =\frac{x}{\mu}, \quad \frac{d\tau}{dt} = \frac{1}{\mu^2}, \quad
\tau\big|_{t=0} = 0,
\quad \frac{\mu_\tau}{\mu} =-\frac{1}{2},
\label{self-similar coordinate 2}
\ee
and the corresponding renormalization
\be 
\rho(t,x) = \frac{1}{\mu^2} \Psi (\tau, y), \quad u(t,x) = \frac{1}{\mu} U(\tau,y), \quad \pi(t,x) = \frac{1}{\mu^2} \Pi(\tau, y).
\label{self-similar coordinate}
\ee
If the initial data of $\mu$ is given by $\mu|_{\tau=0} = \mu_0>0$, then from the equation $\frac{\mu_\tau}{\mu}=-\frac{1}{2}$ determined in \eqref{self-similar coordinate 2}, we can explicitly solve it by
\be
\mu(\tau) = \mu_0 e^{-\frac{1}{2} \tau}, \quad \forall \; \tau\ge 0.
\label{bootstrap assum: scaling}
\ee
The original system \eqref{equation: NS-KS velocity form} of $(\rho, u, \pi)$ is mapped to the following renormalized system of $(\Psi, U, \Pi)$ under the self-similar coordinate,
\be
\begin{cases}
\partial_\tau \Psi + \frac{1}{2} \Lambda \Psi = \Delta \Psi + \na \cdot \left( \Psi \na \Dein \Psi \right) - U \cdot \nabla \Psi,  \\
\partial_\tau U+  \frac{1}{2} \left( U + y \cdot \nabla U \right) +U \cdot \nabla U = \Delta U -\nabla \Pi - \mu \Psi e_3, \\
\nabla \cdot U =0,
\end{cases}
\label{renormalized equation: NS-KS}
\ee
where $\Lambda$ is the scaling operator
\be \Lambda f :=2f +  y\cdot \nabla f.
\label{scaling oper: defi}
\ee
Plugging the ansatz $\Psi = Q + \e$, where $Q$ is the self-similar profile to the Keller-Segel equation given in \eqref{Q: definition}, the system is then transformed into:
\be
\begin{cases}
\partial_\tau \ep = -\calL \ep + \nabla \cdot (\ep \nabla \Dein \ep) -U \cdot \na \Psi, \\
\partial_\tau U+  \frac{1}{2} \left( U + y \cdot \nabla U \right) +U \cdot \nabla U = \Delta U -\nabla \Pi - \mu \Psi e_3,  \\
\nabla \cdot U =0,
\end{cases}
\label{linearized equation}
\ee
where $-\calL$ is the related linearized operator defined by
\be
\calL = \calL_0 + \calL', \quad \text{with }
\calL_0 f= -\Delta f + \frac{1}{2} \Lambda f \;  \text{ and } \; \calL' f = - \na \cdot (f \na \Dein Q) -
\na \cdot (Q \na \Dein f).
\label{linearized operator: def}
\ee

In the following content of Section \ref{section: linear theory}, our focus will be on studying the spectral properties of the linearized operator $-\calL$.

\subsection{Linear theory of $\calL_0$}
\label{subsection: spectral property of L0}
\mbox{}

In this subsection, we study the spectral properties of the linear operator $\calL_0$ in \eqref{linearized operator: def}, which is the dominating term of the linearized operator $\calL$.

Firstly, we study behavior of $\calL_0$ on $L_\omega^2$, where $L^2_\omega(\RR^3)$ is an $L^2$ weighted space defined as follows: 
\bee
(f, g)_{L^2_\omega(\RR^3)} = \int_{\RR^3} f \bar g \omega dy,\quad \omega(y) = e^{-\frac{|y|^2}{4}}. 
\eee

\begin{lemma}[Self-adjointness and semigroup of $\calL_0$ in $L^2_\omega$]\label{lemselfadjcalL0}
The operator $\calL_0: C^\infty_0(\RR^3) \subset L^2_\omega(\RR^3) \to L^2_\omega(\RR^3)$ is essentially self-adjoint. Denote $\calL_0 \big|_{L^2_\omega}$ as its unique self-adjoint extension, then $\sigma\left(\calL_0 \big|_{L^2_\omega}\right) \subset \left[ \frac 14, \infty \right)$, and $-\calL_0 \big|_{L^2_\omega}$ generates a strongly continuous semigroup $(e^{-\tau \calL_0})_{\tau \ge 0}$ of bounded operators on $L^2_\omega(\RR^3)$. Explicitly, for $f \in C^\infty_0(\RR^3)$, 
\be
 \left(e^{-\tau \calL_0} f\right)(y) = e^{-\tau} (G_{1- e^{-\tau}} * f)\left(e^{-\frac{\tau}{2}} y \right), \label{eqsemicalL0}
 \ee
 where $G_\l = \l^{-\frac 32} G_1(\l^{-\frac 12} \cdot)$, with $G_1(y)= (4\pi)^{-\frac 32} e^{-\frac{|y|^2}{4}}$ being the heat kernel. 
\end{lemma}

\begin{proof}
Conjugating $\calL_0$ with the unitary map
\be
 U: L^2(\RR^3) \to L^2_\omega(\RR^3),\quad u \mapsto \omega^{-\frac 12} u, \label{eqsemigroupcalL0}
\ee
we obtain 
\bee
  A_0 := U^{-1} \calL_0 U = -\Delta + \frac{|y|^2}{16} + \frac 14: C^\infty_0(\RR^3) \subset L^2(\RR^3) \to L^2(\RR^3).
\eee
This is a dilated harmonic oscillator, which is essentially self-adjoint (see for example \cite[Example 9.4]{helffer2013spectral}) and hence so is $\calL_0$ on $L^2_\omega$. 
Moreover, 
\[ \sigma\left(\calL_0\big|_{L^2_\omega}\right) = \sigma\left( A_0\big|_{L^2}\right) \subset \left[\frac 14, \infty\right), \]
so from the Hille-Yosida theorem, $-\calL_0$ generates a strongly continuous contraction semigroup on $L^2_\omega$. 

Now we compute \eqref{eqsemicalL0}. For $f_0 \in C^\infty_0(\RR^3)$, the standard semigroup theory (see \cite[Chap. 2 Proposition 6.4]{engel2000one}) implies that $f(\tau) := e^{-\tau \calL_0} f_0$ is the unique mild solution of the Cauchy problem
\be
\begin{cases}
      \pa_\tau f + \calL_0 f  = 0,  \\
    f\big|_{t=0} = f_0,
\end{cases}
\label{eqlinearevocalL0}
\ee
namely $\int_0^\tau f(s) ds \in D\left(\calL_0 \big|_{L^2_\omega}\right)$ for all $\tau \ge 0$ and 
\[ f(\tau) = f_0 -\calL_0 \int_0^\tau f(s) ds. \]

Now let $g(t) = e^{t\Delta} f_0 =G_t * f_0$. Then $f_0 \in C^\infty_0(\RR^3)$ implies $g(t) \in \mathscr{S}(\RR^3)$, and thus $g(t)$ is the classical solution of the linear heat equation
\[ 
\begin{cases}
     \pa_t g -\Delta g  = 0,  \\
    g\big|_{t=0} = f_0.
\end{cases}
\]
Taking the self-similar renormalization
\bee
  \tau(t) = -\ln (1-t),\quad \l(t) = \sqrt{1-t},\quad y = \frac{x}{\l(t)},\quad v(\tau, y) = \l(t)^2 g(t, x),
\eee
we see $v(\tau) \in \mathscr{S}(\RR^3)$ for $\tau \in [0, \infty)$ is a classical solution to \eqref{eqlinearevocalL0}, in particular a mild solution since $\mathscr S(\RR^3)\subset D\left(\calL_0 \big|_{L^2_\omega}\right)$. By the uniqueness of mild solution, we have $f(\tau, y) = v(\tau, y)$, namely
\bee
  \left(e^{-\tau \calL_0} f_0 \right)(y) = \l(t(\tau))^2 \left(G_{t(\tau)} * f_0 \right)\left(\l(t(\tau))y \right) 
  = e^{-\tau} \left(G_{1-e^{-\tau}} * f_0 \right) (e^{-\frac{\tau}2} y).
\eee
\end{proof}

Next, we study the spectral properties of $\calL_0$ on Hilbert space $H^k$.
\begin{lemma}[Closedness and smoothing resolvent estimate for $\calL_0$ in $H^k$]
\label{lemclosmores}
Let $k \ge 0$. The operator $\calL_0: C^\infty_0(\RR^3) \subset H^k(\RR^3) \to H^k(\RR^3)$ is closable and we define the closure as $\left( \calL_0 \big|_{H^k}, \calD\left(\calL_0 \big|_{H^k}\right) \right)$. For any $\l \in \rho\left(\calL_0\big|_{H^k} \right)$ with $\Re \l < \frac 14$, the resolvent satisfies the following  smoothing estimate: for any $m \ge k$,
    \be
      \left\| (\calL_0 \big|_{H^k} - \l)^{-1} f \right\|_{H^{k+\frac 32}} \le C \left(\left|\frac 14 - \Re \l \right|^{-\frac 14} + \left|\frac 14 - \Re \l \right|^{-1} \right) \|f \|_{H^k}, \label{eqresest}
    \ee
    for some $C=C(k)>0$. In particular, $\calD \left(\calL_0 \big|_{H^k}\right) \subset H^{k+ \frac{3}{2}}$.
\end{lemma}

\begin{proof}

\noindent Firstly, we consider the coercivity of $\calL_0$ in $H^k$ with domain $D(\calL_0) = C_0^\infty(\RR^3)$. For any $f \in C_0^\infty(\RR^3)$, the $\dot H^k$ norm of $f_\lambda(y) = \lambda^2 f(\lambda y)$ satisifies
\[
\| f_\lambda \|_{\dot H^k}^2 = \int_{\RR^3} |D^k f_\lambda|^2 dy  = \lambda^{2k+1} \| D^k f \|_{L^2}^2,
\]
then taking derivative with respect to $\lambda$ and choosing $\lambda=1$, then
\be
 \Re (D^k (\Lambda f), D^k f)\Ltwo = \frac{2k+1}{2} \| f \|_{\dot H^k}^2.
\label{energy estimate: scaling part KS}
\ee
Hence
 {
\begin{align}
\Re (\calL_0 f, f)_{H^k} 
& =  \Re \left(-\Delta f + \frac{1}{2} \Lambda f, f \right)_{L^2}  + 
\Re \left( D^k \left( -\Delta f + \frac{1}{2} \Lambda f \right), D^k f \right)_{L^2} \notag \\
& \ge \frac{1}{4} \| f \|_{H^{k+1}}^2 \ge \frac{1}{4} \| f \|_{H^{k}}^2.
\label{coercivity of L0}
\end{align}
}
Then applying Cauchy-Schwarz inequality to (\ref{coercivity of L0}), we easily obtain $\| \calL_0 f \|_{H^k} \ge \frac{1}{4} \| f \|_{H^k}$ for $f \in C^\infty_0(\RR^3)$. Also since $\calL_0$ is densely defined on $H^k$, it is closable according to \cite[Chap. II,  Proposition 3.14 (iv)]{engel2000one}.
Besides, we have $ \calD\left(\calL_0 \big|_{H^k}\right) \subset \calD\left(\calL_0 \big|_{L^2_\omega}\right)$, then $\left(\calL_0 \big|_{H^k} - \l\right)^{-1} g = \left(\calL_0 \big|_{L^2_\omega} - \l\right)^{-1} g$ for $g \in H^k$, and $\calL_0 \big|_{H^k} f =  \calL_0 \big|_{L^2_\omega} f$ for $f \in \calD\left(\calL_0 \big|_{H^k}\right)$, which follows that $H^k$ is embedded in $L^2_\omega$ and the uniqueness of closure $\left(\calL_0, C_0^\infty(\RR^3) \right)$ on $L_\omega^2$ and $H^k$.

From the closedness of $\calL_0 \big|_{H^k}$, it suffices to prove  \eqref{eqresest} for $f \in C^\infty_0(\RR^3)$, for which we can replace $\calL_0 \big|_{H^k}$ by $\calL_0 \big|_{L^2_\omega}$ and apply \eqref{eqsemicalL0}. Indeed, noticing that for any $\zeta, \mu > 0$,
\[
\| G_\zeta * f  \|_{H^{k+\frac{3}{2}}} = C \| \la \cdot \ra^{k+\frac{3}{2}} \hat G_{\zeta} \hat f \|_{L^2} 
\le C \| \la \cdot \ra^{\frac{3}{2}} \hat G_\zeta \|_{L^\infty} \| \la \cdot \ra^k f\|_{L^2} \le C \la \zeta^{-1} \ra^{\frac{3}{4}}  \| f \|_{H^k},
\]
for some $C=C(k)>0$, thus for $\Re\l < \frac 14$ and $f \in C^\infty_0(\RR^3)$, we obtain via Laplace transform and \eqref{eqsemicalL0}  that
\bee
 (\calL_0 - \l)^{-1} f = \int_0^\infty e^{-\tau (\calL_0 - \l)} f d\tau = \int_0^\infty e^{-(1-\l)\tau} (G_{1-e^{-\tau}} * f )\left(e^{-\frac \tau 2} \cdot \right) d\tau.
\eee
This implies that there exists a constant $C=C(k)>0$ such that
\bee
&&\left\| (\calL_0 - \l)^{-1} f  \right\|_{H^{k+\frac 32}} \le \int_0^\infty e^{-(1-\l)\tau} \left\| (G_{1-e^{-\tau}} * f )\left(e^{-\frac \tau 2} \cdot \right) \right\|_{H^{k+\frac 32}} d\tau \\
&\le & C\int_0^\infty e^{-(\frac 14 - \l) \tau} \left\la \frac{1}{1-e^{-\tau}} \right\ra^{\frac 34} \| f \|_{H^k}d\tau 
\le C\left(\left|\frac 14 - \Re \l \right|^{-\frac 14} + \left|\frac 14 - \Re \l \right|^{-1} \right)  \| f \|_{H^k}.
\eee
\end{proof}

\subsection{Perturbed maximal dissipativity of $-\calL$}
\label{subsection: accretivity and maximality of L}
This subsection focuses on the analysis of the linearized operator $-\calL$. The main result is the following:

\begin{proposition}[Perturbed maximal dissipativity of $-\calL$] For $k \ge 0$,
\label{prop: decomposition of L}
    there exist a maximally dissipative operator $A_0:\calD(A_0) \subset H^k(\RR^3;\CC) \to H^k(\RR^3;\CC)$ with $\calD(A_0) = \calD\left( \calL_0 \big|_{H^k} \right)$, namely 
    \bea
     \forall f \in \calD\left( \calL_0 \big|_{H^k} \right), \quad  \Re (A_0 f, f)_{H^k} \le 0, \label{eqA0dis} \\
      \exists R > 0, \quad A_0 - R: \calD\left( \calL_0 \big|_{H^k} \right) \to H^k(\RR^3;\CC)\,\,{\rm is \,\,surjective;} \label{eqA0max}
    \eea 
    and $K$ compact on $H^k(\RR^3; \CC)$ such that
    \be
    -\calL = A_0 - \frac 1{16} + K.
    \label{L: decomposition}
    \ee
\end{proposition}

\begin{remark}
    We recall that maximal dissipative operators are closed. Hence $\calD(\calL \big|_{H^k}) = \calD\left( \calL_0 \big|_{H^k} \right)$, and $\left(\calL \big|_{H^k}, \calD(\calL \big|_{H^k})\right)$ is closed on $H^k$.
\end{remark}

\mbox{}

Recall $\calL'$  from \eqref{linearized operator: def}. We begin by showing its compactness. Decompose $\calL'$ as
\begin{align}
    & \calL '  = \calL'_1+ \calL'_2 + \calL'_3,  \notag \\
    & \calL_1' f = -2 Qu, \quad \calL_2' f = - \nabla Q \cdot \nabla \Dein f, \quad \calL_3' f = - \na \Dein Q \cdot \nabla f.
\label{linarized operator: perturbation term}
\end{align}
We also define $\calC_{3, k}'$ by
\be
\calC_{3,k}' = [D^k, \calL_3'].
\label{commutator: definition}
\ee

\begin{lemma}[Compactness of  $\calL_1', \calL_2'$, $\calL_3'$ and $ \calC_3'$]
\label{lemma: compactness of the operator}
For any $k \ge 0$, the linear operators $\calL_1': H^{k+1}(\RR^3) \to H^k(\RR^3)$, $\calL_2': H^{k}(\RR^3) \to H^k(\RR^3)$, $\calL_3': H^{k+\frac{3}{2}}(\RR^3) \to H^k(\RR^3)$ and $\calC_{3,k}': H^{k+1}(\RR^3) \to L^2(\RR^3)$ are all compact.
\end{lemma}

\begin{proof}
    Firstly of all, we consider the compactness of $\calL_1':H^{k+1} \to H^k$. In fact, for any $R>0$, since $2\chi_R Q \in  C_0^\infty(\RR^3) \subset \mathscr{S}(\RR^3)$, by \cite[Theorem $1.68$]{MR2768550}, we see that the multiplier $f \mapsto \calL_{1,R}' f : = 2\chi_R Qf$ is compact from $H^{k+1} \to H^k$. In addition, there exists $C=C(k)>0$ such that
    \begin{align*}
        \| (1-\chi_R) Q f \|_{H^k} 
        & = \| (1-\chi_R) Q f \|_{L^2} + \sum_{|\alpha|=k}\| \partial^\alpha \left((1-\chi_R) Q f \right) \|_{L^2} \\
        &  \le C \| (1-\chi_R) Q \|_{W^{k,\infty}} \| f \|_{H^k} \le C R^{-p} \| f \|_{H^k}, \quad \forall f \in H^k,
    \end{align*}
    for some $p>0$ uniformly in $R>0$. Then by the standard diagonal method, $\calL_1'$ is compact from $H^{k+1}$ to $H^k$.

    Secondly, for the commutator term $\calC_{3,k}'$ defined in \eqref{L: decomposition}, it involves at most the $k$-th order derivative, then with the decay of the arbitrary derivative of $\na Q$, we can repeat the argument above to conclude that $\calC_{3,k}'$ is compact from $H^{k+1} \to H^k$. 

    Thirdly, we consider the compactness of $\calL_3': H^{k+\frac{3}{2}} \to H^k$. Recall the definition of $\calL_3'$ in \eqref{L: decomposition}, it suffices to prove that the multiplier  $\calL_3''f: = \na \Dein Q f$ is compact from $H^{k+\frac{1}{2}} \to H^k$. In fact, similar to the argument above, by \cite[Theorem $1.68$]{MR2768550} again, we can see that $\calL_{3,R}'' = \chi_R \na \Dein Q$ is a compact multiplier from $H^{k+\frac{1}{2}}$ to $H^k$. In addition, there exists a constant $C=C(k)>0$ such that
    \[
    \Big\| \left((1-\chi_R) \na \Dein Q \right) f \Big\|_{H^k} \le C \| (1-\chi_R) \na \Dein Q \|_{W^{k,\infty}} \| f \|_{H^k} \le C R^{-p} \| f \|_{H^k}, \; \forall f \in H^k,
    \]
    for some $p>0$ uniformly in $R>0$. Then by the standard diagonal argument again, $\calL_3'$ is compact from $H^{k+\frac{3}{2}}$ to $H^k$.

    Finally, let us consider the compactness of the nonlocal term $\calL_2':H^{k} \to H^k$. Then there exists a uniform constant $C=C(k)>0$, such that for any $R>0$, by H\"{o}lder's inequality and Sobolev inequality,
    \begin{align}
        & \quad \Big\| (1-\chi_R) \na Q \cdot \na \Dein f  \Big\|_{H^k} \notag \\
        & \le C \Big\| (1-\chi_R) \na Q \cdot \na \Dein f  \Big\|_{L^2} + C\sum_{|\alpha| =k} 
         \Big\| \partial^\alpha \left( (1-\chi_R) \na Q  \right) \cdot \na \Dein f  \Big\|_{L^2} \notag \\
         & \qquad  \qquad \qquad + C \sum_{|\alpha| < k, |\alpha| + |\beta| =k} \Big\| \partial^\alpha \left( (1-\chi_R) \na Q  \right) \cdot \na \Dein \partial^\beta f  \Big\|_{L^2} \notag \\
         & \le C \Big\| (1-\chi_R) \na Q \Big\|_{W^{k,3}} \| \na \Dein f \|_{L^6} + 
         C \Big\| (1-\chi_R) \na Q \Big\|_{W^{k,\infty}} \| \na \Dein f \|_{\dot H^1 \cap \dot H^k}  \notag\\
         & \le
         \begin{cases}
            C R^{-p} \| f \|_{H^{k-1}}, &  k \ge 1, \\
            C R^{-p} \|f \|_{L^2}, & k=0,
         \end{cases} \quad \forall f \in H^k,
         \label{nonlocal term}
    \end{align}
    for some $p>0$ independent of $R > 0$. By similar computation and $\nabla Q \in W^{k, 3} \cap W^{k,\infty}$, we have the boundedness 
    \[
    \| \chi_R \calL_2' f \|_{H^{k+1}} \le C \| f \|_{H^{k}}, \quad \forall f \in H^{k}, \,\, \forall R > 0.
    \]
    Then by Rellich-Kondrachov compactness theorem \cite[Theorem $9.16$]{MR2759829}, the operator $\chi_R \calL_2'$ is compact from $H^k$ to $H^k$. Consequently, using the diagonal argument again, we obtain that $\calL_2': H^k \to H^k$ is compact.
\end{proof}

\begin{lemma}[Almost coercivity for $\calL'$] For any $k \in \NN$, and any $\delta >0$, there exist finite $(q_j)_{1 \le j \le N_0} \subset C_0^\infty(\RR^3;\CC)$ such that for $f \in H^{k+1}(\RR^3; \CC)$ with $f \perp_{H^k} q_j$ for $1 \le j \le N_0$, we have
\be
    \Re (\calL'f,f)_{H^k} \ge 
    -\delta \| f \|_{H^{k+1}}^2.
    \label{L': smallness}
    \ee
\end{lemma}

\begin{proof}
    By Lemma \ref{lemma: compactness of the operator}, $\calL_1' + \calL_2'$ is compact from $H^{k+1}$ to $H^k$, then  {it yields a finite-rank operator $\tilde \calT_{1,2,\delta}: H^{k+1} \to H^k$}, such that $\tilde \calR_{1,2,\delta}= \calL_1' + \calL_2' - \tilde \calT_{1,2,\delta}:H^{k+1} \to H^k$ is bounded with $\| \tilde \calR_{1,2,\delta} \|_{H^{k+1} \to H^k} \le \frac{\delta}{8}$. In particular, there exists $\{{\tilde {q}}_j \}_{1 \le j \le N_{1,2}} \subset H^{k+1}$ and $\{ p_j \}_{1 \le j \le N_{1,2}} \subset H^k$ such that $\tilde \calT_{1,2,\delta}$ can be represented by
	\[
	\tilde \calT_{1,2,\delta} = \sum_{j=1}^{N_{1,2}} \left( \cdot , {\tilde {q}} _j \right)_{H^{k+1}} p_j.
	\]
  By the density of $C_0^\infty$ in $H^{k+1}$, we introduce $\{\tilde {\tilde q}_j\}_{1 \le j \le N_{1,2}} \subset C_0^\infty$ such that $\tilde {\tilde{\calT}}_{1,2,\delta} = \sum_{j=1}^{N_{1,2}} (\cdot, {\tilde {\tilde{q}}}_j)_{H^{k+1}} p_j$ satisfies
 \[
 \| \tilde \calT_{1,2,\delta} -\tilde {\tilde{\calT}}_{1,2,\delta} \|_{H^{k+1} \to H^k} \le \frac{\delta}{8}.
 \]
Then by integration by parts, Riesz's representative theorem and the density of $C_0^\infty$ in $H^k$, we can find $\{q_j\}_{1\le j \le N_{1,2}} \subset C_0^\infty \subset H^k$ such that the finite-rank operator $\calT_{1,2,\delta} = \sum_{j=1}^{N_{1,2}} (\cdot,q_j)_{H^k} p_j$ satisfies
\[
\| \calT_{1,2,\delta} -  \tilde {\tilde{\calT}}_{1,2,\delta}\|_{H^{k+1} \to H^k} \le \frac{\delta}{8}.
\]
Hence, combining all of the estimates above and together with the  triangle inequality, $\calR_{1,2,\delta} = \calL_1' + \calL_2' - \calT_{1,2,\delta}$ satisfies 
\[
\| \calR_{1,2,\delta} \|_{H^{k+1} \to H^k}  \le \frac{\delta}{2}.
\]

Similarly, since $\calC_{k,3}'$ is compact from $H^{k+1}$ to $L^2$ derived from Lemma \ref{lemma: compactness of the operator}, repeating the argument above, we can find sequences $(q_j)_{N_{1,2}+1}^{N_{1,2}+ N_3} \subset C_0^\infty$ and $(p_j)_{N_{1,2}+1}^{N_{1,2}+ N_3} \subset L^2$, such that
\[
\calT_{3,\delta} = \sum_{N_{1,2}+1}^{N_{1,2}+ N_3} (\cdot, q_j)_{H^k} p_j,
\]
and the residual $ \calR_{3,\delta} = \calT_{3,\delta} -\calC_{3,k}'$ satisfies $\| \calR_{3,\delta} \|_{H^{k+1} \to L^2}\le \frac{\delta}{2}$. 

Additionally, we have the coercivity of $\calL_3'$ under the $L^2$ inner product as follows:
\[
\Re (\calL_3'  f,  f)_{L^2} = -\Re \int \left( \na \Dein Q \cdot \nabla f \right) \overline{f} dy =\frac{1}{2} \int Q | f|^2 dy \ge 0. 
\]
Consequently, combining all the argument above, for any $f\in H^{k+1}$ satisfying $f \perp_{H^k} q_j$ for any $1 \le j \le N_0=N_{1,2} +N_3$, we have $\calT_{1,2,\delta} f = \calT_{3,\delta} f = 0$, hence
\begin{align*}
    & \quad \Re (\calL'f,f)_{H^k(\RR^3)}
    = \Re ((\calL_1' + \calL_2')f,f)_{H^k(\RR^3)} + \Re(\calL_3' f,f)_{L^2(\RR^3)} + \Re(D^k \calL_3' f, D ^k f)_{L^2(\RR^3)}  \\
    & = \Re ((\calL_1' + \calL_2')f,f)_{H^k(\RR^3)} + \Re(\calC_{3,k}' f, D^k f)_{L^2(\RR^3)} + \Re(\calL_3' f,f)_{L^2(\RR^3)} + \Re(\calL_3' D^k f, D^k f)_{L^2(\RR^3)}  \\
    & \ge \Re ((\calL_1' + \calL_2')f,f)_{H^k(\RR^3)} + \Re(\calC_{3,k}' f, D^k f)_{L^2(\RR^3)} \\
    & =  \Re (\calR_{1,2,\delta}f,f)_{H^k(\RR^3)} + \Re(\calR_{3,\delta} f, D^k f)_{L^2(\RR^3)}
    \ge -\delta\| f\|_{H^{k+1}}^2.
\end{align*}
\end{proof}

\mbox{}

Now we are in place of proving Proposition \ref{prop: decomposition of L}.

\begin{proof}[Proof of Proposition \ref{prop: decomposition of L}] \underline{\textit{Step 1. Dissipativity of $-\calL$ modulo finite dimensions.}} We claim that there exist $\{q_i\}_{i=1}^N \subset C_0^\infty(\RR^3)$ with $(q_i, q_j)_{H^k} = \delta_{ij}$, and $C>0$ such that for $f \in \calD(\calL_0 \big|_{H^k})$,
    \be
    \Re (\calL f, f)_{H^k(\RR^3)}  \ge \frac{1}{8} \| f \|_{H^k(\RR^3)}^2 - C \sum_{j=1}^N |(f,q_j)_{H^k}|^2.
    \label{L: coercivity another form}
    \ee

In fact, note that Lemma \ref{lemma: compactness of the operator} and direct estimate of $\calL_3'$ from  \eqref{linarized operator: perturbation term} imply 
\be \| \calL'  \|_{ H^{k+1} \to H^{k}} \le C,
\label{eqestL'}
\ee
for some $C=C(k)>0$. Thus the LHS of \eqref{L: coercivity another form} is well-defined for $u \in \calD(\calL_0 \big|_{H^k}) \subset H^{k+\frac 32}$ (from Lemma \ref{lemclosmores}). Applying (\ref{coercivity of L0}) and (\ref{L': smallness}) with $\delta=\frac{1}{16}$, there exists $(q_{j})_{j=1}^{N} \subset C_0^\infty(\RR^3)$ such that
    \be
    \Re (\calL f, f )_{H^k(\RR^3)} \ge \frac{3}{16} \| u \|_{H^{k+1}(\RR^3)}^2,
    \label{coercivity of L}
    \ee
    for any $f$ satisfying  $f \perp_{H^{k}} q_j, \; \forall\;  1 \le j \le N$. Without loss of generalization, we assume that the $\{q_j\}$ obtained previously satisfies the orthogonal condition $(q_i, q_j)_{H^k} =\delta_{ij}$, otherwise we can use the Gram-Schmidt orthonormalization process to realize it.

Thus for any $f \in H^k$, if we decompose $f$ as
\[
f= f_1 + f_2, \quad \text{ where } f_2 =  \sum_{j=1}^{N} (f, q_j)_{H^k} q_j, \;  f_1 \in \text{span} \{ q_1, q_2, ..., q_{N} \}^\perp, \text{ and } f_1 \perp_{H^k} f_2,
\]
then by Cauchy-Schwarz inequality, Young's inequality and the fact that $q_j \in C_0^\infty(\RR^3)$, we conclude that
\begin{align*}
    \Re (\calL f,f)_{H^k}
    & = \Re (\calL f_1 ,f_1)_{H^k} + \Re (\calL f_1 ,f_2)_{H^k}
    + \Re(\calL f_2 ,f_1)_{H^k} + \Re(\calL f_2,f_2)_{H^k} \\
    & \ge \frac{3}{16} \| f \|_{H^k}^2 - \frac{1}{16}\| f_1\|_{H^k}^2
    - C\left( \| \calL f_2\|_{H^k}^2 +  \| \calL^* f_2\|_{H^k}^2 +\| f_2\|_{H^k}^2    \right) \\
    & = \frac{1}{8} \| f \|_{H^k}^2 - C\left( \| \calL f_2\|_{H^k}^2 +  \| \calL^* f_2\|_{H^k}^2 +\| f_2\|_{H^k}^2    \right)  \\ 
    & \ge \frac{1}{8} \| f \|_{H^k}^2 - C \sum_{j=1}^N |(f,q_j)_{H^k}|^2,
\end{align*}
for some universal $C>0$, which is definitely \eqref{L: coercivity another form}. 

\mbox{}

\noindent \underline{\textit{Step 2. Construction of $A_0$ and end of proof.}} With $(q_j)_{1 \le j \le N}$ and $C>0$ from Step 1, we define $A_0$ and $K$ as
\[
K= C \sum_{j=1}^N (\cdot,q_j)_{H^k} q_j, \quad  \quad A_0 = -\calL + \frac{1}{16} - K: \calD(\calL_0 \big|_{H^k}) \to H^k(\RR^3; \CC).
\]
Immediately, $K$ is compact on $H^k(\RR^3;\CC)$ and $A_0$ is well-defined on $\calD(\calL_0 \big|_{H^k})$ due to \eqref{eqestL'} and Lemma \ref{lemclosmores}. The decomposition \eqref{L: decomposition} comes from the definition of $A_0$, and the dissipativity follows (\ref{L: coercivity another form}) directly. 

It remains to prove maximality \eqref{eqA0max}. Noticing that
\[
-A_0 -\lambda= \calL_0 + \calL' - \frac{1}{16} +K -\lambda = \left( I + (\calL'+K) \left(\calL_0 - \frac{1}{16} - \lambda \right)^{-1} \right) \left(\calL_0 - \frac{1}{16} - \lambda \right),
\]
then from resolvent estimate (\ref{eqresest}) and boundedness of $\calL'$ \eqref{eqestL'}, we can choose $\lambda \ll -1$ such that $\big\| (\calL'+K) \left(\calL_0 - \frac{1}{16} - \lambda \right)^{-1} \big\|_{H^k \to H^k} \ll 1$. This implies $\left( I + (\calL'+K) \left(\calL_0 - \frac{1}{16} - \lambda \right)^{-1} \right)$ is invertible on $H^k(\RR^3;\CC)$, and thus $(-A_0-\lambda)^{-1}: H^k(\RR^3; \CC) \to D\left( \calL\big|_{H^k}\right)$ is well-defined and bounded, which yields \eqref{eqA0max}, then we have concluded the proof.
\end{proof}

\mbox{}

\subsection{Spectral properties of $-\calL$}
\label{subsection: spectral of L}

As a result of the decomposition (\ref{L: decomposition}), we conclude the following spectral properties of $-\calL$ via the abstract semigroup theory from \cite{MR4359478}.

\begin{proposition}[{\cite[Lemma 3.3, Lemma 3.4]{MR4359478}} Spectral properties of $-\calL$ on $H^k(\RR^3; \CC)$]
\label{proposition: spectral properties of L}
    For $k \ge 0$, consider $-\calL: D(\calL\big|_{H^k}) \subset H^k(\RR^3; \CC) \to H^k(\RR^3; \CC)$ given in (\ref{linearized operator: def}) and any $\delta \in (0, \frac{1}{16})$. There exists $0 < \delg < \delta$ such that the following spectral properties of $-\calL$ hold\footnote{The choice of $\delta_g$ is similar to the proof of \cite[Lemma 3.4]{MR4359478}. For instance, we can take
    \[\epsilon_g = \inf \left( \left\{ \Re \l + \frac{\delta} 2: \l \in \sigma(-\calL), 0 > \Re \l >  -\frac{\delta}{2}  \right\} \cup \left\{\frac\delta 2 \right\} \right) \in \left(0, \frac \delta 2\right], \]
    to be the length of the spectral gap on the right of $-\frac \delta 2$, and then let $-\frac \delg 2 = -\frac \delta 2 + \frac {\epsilon_g}2$.}:

    \noindent $\bullet$ The set $\Lambda_{\calL} = \sigma(-\calL) \cap \{ \lambda \in \CC: \Re \lambda > - \frac{\delg}{2} \}$ is finite and formed only by eigenvalues of $-\calL$. Furthermore, the generalized eigenspace
    \[
    V = \bigoplus_{\lambda \in \Lambda_{\calL}} \ker (-\calL -\lambda )^{\mu_\lambda},
    \]
    is finite-dimensional, where $\mu_\lambda$ is the smallest integer such that $\ker (-\calL -\lambda  )^{\mu_\lambda} = \ker (-\calL -\lambda  )^{\mu_\lambda +1}$.

\noindent $\bullet$ Consider $-\calL^*= A_0^* -\frac{1}{16}+ K^*$ and $\Lambda_{\calL}^* = \sigma(-\calL^*) \cap \{ \lambda \in \CC: \Re \lambda > - \frac{\delg}{2} \}$. Similarly, we define the corresponding generalized eigenspace by
\[
V^* = \bigoplus_{\lambda \in \Lambda_{\calL}^*} \ker (-\calL^* - \lambda)^{\mu_\lambda^*},
\]
with $\Lambda_{\calL}^* = \bar\Lambda_{\calL}$ and $\mu_\lambda = \mu_{\bar \lambda}^*$. Let $V^{*\perp}$ be the orgothogonal complement of $V^*$ in $H^k(\RR^3; \CC)$. Then $V$ and ${V^*}^\perp$ are invariant under $-\calL$. And the whole space $H^k(\RR^3; \CC)$ can be decomposed by $H^k(\RR^3; \CC) = V \oplus {V^*}^\perp$.

\noindent $\bullet$ Define the spectral projection of $-\calL$ to the set $\Lambda_{\calL}$:
\be
P_u = \frac{1}{2\pi i}\int_\Gamma (z - (-\calL))^{-1} dz,
\label{Pu: definition}
\ee
where $\Gamma \subset \rho(-\calL)$ is an arbitrary anti-clockwise contour containing the set $\Lambda_{\calL}$. Then $V = {\rm Ran} P_u$ and $V^{*\perp} = \ker P_u = {\rm Ran} (1-P_u)$.\footnote{This statement is contained in the proof of 
\cite[Lemma 3.3]{MR4359478}.}

\noindent $\bullet$ (Stable part) $-\calL$ generates a semigroup, which we denote by $e^{-\tau \calL}$, and it satisfies
\be
\| e^{-\tau \calL} v\|_{H^k} \le C e^{-\frac{1}{2} \delg \tau} \| v \|_{H^k}, \quad \forall v \in V^{* \perp},\quad \tau \ge 0,
\label{decay rate: stable part}
\ee
for some constant $C=C(k)>0$ uniformly in $v \in V^{*\perp}$ and $\tau \ge 0$.
\end{proposition}

Notice that $\calL$ is real-valued operator, we have \be \overline{\calL f} = \calL \bar f,\quad f \in \calD\left( \calL\big|_{H^k(\RR^3; \CC)} \right),
\label{eqsymcalL}
\ee 
where $\bar f$ is the complex conjugation of $f$. 
This symmetry implies the following structure of unstable spectrum $\Lambda_{\calL}$.

\begin{proposition}[Structure of unstable spectrum on $H^k(\RR^3; \RR)$]
\label{Proposition: unstable mode}
For $k \ge 0$, and any $\delg>0$ determined in Proposition \ref{proposition: spectral properties of L}, the following statements of  $-\calL$ are true.
\mbox{}

  \begin{enumerate}
      \item Symmetry of unstable spectrum: $\Lambda_{\calL}$ is symmetric with respect to the real axis, namely we have
      \be  \Lambda_{\calL} \cap \RR = \{ \l_j\}_{j=1}^N,\quad \Lambda_{\calL} \backslash \RR = \{ \xi_j, \bar \xi_j\}_{j=1}^J. \label{spectral set of L} \ee
      Moreover, \textcolor{black}{there exists $\{ v_j \}_{j=1}^{M_1} \subset H^k(\RR^3; \RR) $ 
      such that
      \[
      \bigoplus_{\lambda \in \Lambda_{\calL} \cap \RR} \ker (\calL - \lambda )^{\mu_\lambda} = \text{span}_\CC \{ v_j \}_{j=1}^{M_1}.
      \]
      }

      \item Symmetry of Riesz projection on $H^k(\RR^3; \RR)$: There exists $\varphi_j, \psi_j \in H^k(\RR^3;\RR)$ for $j = 1, ..., N$ such that
      \be (\varphi_i, \psi_{j})_{H^k} = \delta_{i,j}, \label{eqorthovarphipsi}
      \ee
      and for real-valued function $f \in H^k(\RR^3; \RR)$,
      \be P_u f = \sum_{i = 1}^N (f, \psi_i)_{H^k} \varphi_j \in H^k(\RR^3; \RR). \label{eqPureal}
      \ee

      \item Stable/unstable decomposition of $H^k(\RR^3;\RR)$: Define 
      \[ H^k_u := {\rm Ran} \left(P_u\big|_{H^k(\RR^3;\RR)}\right),\quad
      H^k_s := {\rm Ran} \left((1-P_u)\big|_{H^k(\RR^3;\RR)}\right).  
      \]
      Then $H^k_u = V \cap H^k(\RR^3; \RR)$, $H^k_s =V^{*\perp} \cap H^k(\RR^3; \RR)$; they are invaraint under $-\calL$; and thus
      $H^k(\RR^3;\RR)$ can be decomposed by
      $
      H^k(\RR^3;\RR) = H^k_u \oplus H^k_s$.
      
      \item Real Jordan normal form of $(-\calL)\big|_{H^k_u}$:  The linear transformation $-\calL|_{H^k_u}$ has all its eigenavlues lying in $\{ z \in \CC: \Re z > -\frac{\delg}{2} \}$. Moreover, there exists a basis of $H^k_u$ such that it can be expressed as a real matrix $J = \text{diag} \{ J_1, J_2,..., J_l \}$,
and here $J_i$ must be one the following two forms:
\begin{align}
    J_i^{Re}=
    \begin{bmatrix}
    \lambda_i & \frac{\delg}{10} & & \\
    & \lambda_i & \ddots & \\
    & & \ddots & \frac{\delg}{10} \\
    & & & \lambda_i
    \end{bmatrix} 
    \quad
    \text{ or }
    \quad 
     J_i^{Im}=
    \begin{bmatrix}
    C_i & \frac{\delg}{10} I_2 & & \\
    & C_i & \ddots & \\
    & & \ddots & \frac{\delg}{10} I_2 \\
    & & & C_i,
    \end{bmatrix},
    \label{Jordan normal form}
\end{align}
where $C_i = \begin{bmatrix}
    a_i & - b_i \\
    b_i & a_i
\end{bmatrix}$
and $I_2 = \begin{bmatrix}
    1 & 0 \\ 0 & 1
\end{bmatrix}$ are the $2 \times 2$ real matrices, where $\l_i, a_i > -\frac{\delta_g}{2}$. 
In particular, we have
\be
w^T J w \ge - \frac{6 \delg }{10} |w|^2, \quad w \in \RR^N.
\label{increase rate: unstable part}
\ee
  \end{enumerate}    
\end{proposition}

\begin{proof} 
(1) The symmetry of $\Lambda_{\calL}$ and thereafter the decomposition \eqref{spectral set of L} follow from that $(\calL -
\l)^n \varphi_\l = 0$ if and only if $(\calL - \bar \l)^n \bar \varphi_\l = 0$ due to  the symmetry of $\calL$ given in \eqref{eqsymcalL}. Taking a (complex-valued) basis $\{\tilde v_j \}_{1 \le j \le M_1} \subset H^k(\RR^3;\CC)$ of the real unstable generalized eigenspace $\bigoplus_{\lambda \in \Lambda_{\calL} \cap \RR} \ker(-\calL -\lambda)^{\mu_\lambda}$, the symmetry \eqref{eqsymcalL} implies that $\Re \tilde v_j, \Im \tilde v_j \in \bigoplus_{\lambda \in \Lambda_{\calL} \cap \RR} \ker(-\calL -\lambda)^{\mu_\lambda}$, from which we can choose a basis $\{ v_j \}_{j=1}^{M_1} \subset H^k(\RR^3; \RR)$.

\mbox{}

(2) Denote the Riesz projection onto $\bigoplus_{\lambda \in \Lambda_{\calL} \cap \{\pm \Im z >0 \}} \ker (\calL -\lambda)^{\mu_{\lambda}}$ by $P_{Im\pm}$ and the one onto $\bigoplus_{\lambda \in \Lambda_{\calL} \cap \RR} \ker (-\calL -\lambda)^{\mu_\lambda}$ by $P_{Re}$. By the definition of $P_{Im+}$ through the contour integration,
      \[
      P_{Im+} = \frac{1}{2 \pi i} \int_{\gamma_2} (\lambda-(-\calL))^{-1} d\lambda,
      \] 
      where $\gamma_2 \subset \rho(-\calL)$ is an anti-clockwise closed curve enclosing $\Lambda_{\calL} \cap \{ \Im z >0 \}$. So for real-valued $f$, 
     \begin{align}
       &\overline {P_{Im+}f} 
      = \overline {\frac{1}{2 \pi i} \int_{\gamma_2} (\lambda-(-\calL))^{-1} f d\lambda }
      = -\frac{1}{2 \pi i} \int_{\gamma_2} (\bar \lambda - (-\calL))^{-1} f d \bar \lambda \notag\\
      =& -\frac{1}{2 \pi i} \int_{\bar \gamma_2} (\lambda - (-\calL))^{-1}f  d \lambda= \frac{1}{2 \pi i} \int_{-\bar \gamma_2} (\lambda - (-\calL))^{-1} f d \lambda = P_{Im-} f
       \label{contour integral}
      \end{align}
    where $-\bar \gamma_2 \subset \rho(-\calL)$ is an anti-clockwise closed curve with its interior only containing $\Lambda_{\calL} \cap \{ \Im z <0 \}$ thanks to the symmetry of $\Lambda_{\calL}$. Similarly, we have  $\overline{P_{Re}f} = P_{Re}f$. 
      
    Now take $\{ w_j \}_{j=1}^{M_2}$ to be the basis of $  \bigoplus_{\lambda \in  \Lambda_{\calL} \cap \{z \in \CC: \Im z >0\}} \ker (-\calL -\lambda)^{\mu_\lambda}$, then $\Re w_j, \Im w_j$ are non-trivial since $\l \notin \RR$. By Riesz representative theorem, there exist $\{ h_j \}_{j=1}^{M_1}, \{ g_j\}_{j=1}^{M_2} \subset H^k(\RR^3; \CC)$ such that
      \[
      P_{Re} = \sum_{j=1}^{M_1} (\cdot, h_j)_{H^k} v_j \quad \text{ and } \quad P_{Im+} = \sum_{j=1}^{M_2} (\cdot, g_j)_{H^k} w_j,
      \]
      with $\{ v_j\}_{j=1}^{M_1} \subset H^k(\RR^3; \RR)$ from (1). Note that the symmetries \eqref{contour integral} and $\overline{P_{Re}f} = P_{Re}f$ also indicate that $P_{Im-} = \sum_{j=1}^{M_2} (\cdot, \bar g_j)_{H^k} \bar w_j$ on $H^k(\RR^3;\RR)$ and $h_j \in H^k(\RR^3;\RR)$. So
      the projection property of $P_{Re}$ and $P_{Im+}$ indicate that
      \be
      (v_i,h_j)_{H^k} = (w_i,g_j)_{H^k} =\delta_{i,j}, \; (v_i, g_j)_{H^k} = (w_i,h_j)_{H^k}= (\bar w_i, g_j)= 0, \quad \forall i,j.
      \label{orthogonal condition: basis}
      \ee
      As a corollary, $\Re h_j, \Re g_j, \Im g_j$ are non-trivial from the non-triviality of $v_i, \Re w_i$ and $\Im w_i$. 
      So with the symmetry of $P_{Im\pm}$ and $P_{Re}$, we compute for any real-valued function $f \in H^k(\RR^3;\RR)$,
      \begin{align*}
      	P_u f &= P_{Re}f  + P_{Im+}f + P_{Im-} f
      	= \frac{1}{2} \left( P_{Re} + \overline{P_{Re}} \right)f + \left( P_{Im+} + P_{Im-} \right) f \\
      	& =\sum_{j=1}^{M_1} (f, \Re h_j)_{H^k} v_j + \sum_{j=1}^{M_2} (f, g_j)_{H^k} w_j +  \sum_{j=1}^{M_2} (f, \bar g_j)_{H^k} \bar w_j \\
      	& = \sum_{j=1}^{M_1} (f, \Re h_j)_{H^k} v_j + 2\sum_{j=1}^{M_2} \left[(f, \Re g_j)_{H^k} \Re w_j + (f, \Im g_j)_{H^k} \Im w_j \right].
       \end{align*}
       This implies the choice of $\{ \varphi_j\}_{j=1}^N, \{ \psi_j\}_{j=1}^N \subset H^k(\RR^3;\RR)$ with $N = M_1 + 2M_2$ to realize \eqref{eqPureal}, and the orthonormal relation \eqref{eqorthovarphipsi} follows \eqref{orthogonal condition: basis}.

\mbox{}
 
 (3) For any $f \in H_u^k = \text{Ran}(P_u\big|_{H^k(\RR^3; \RR)}) \subset V$, there exists $g \in H^k(\RR^3; \RR)$ such that $f =P_u g$, then by (\ref{eqPureal}), $f \in H^k(\RR^3; \RR)$, hence $f \in V \cap H^k (\RR^3; \RR)$. Conversely, for any $f \in V \cap H^k(\RR^3;\RR)$, then by (\ref{eqPureal}) again, $f=P_u f \in H_u^k$. Hence $H_u^k = V \cap H^k(\RR^3;\RR)$. Likewisely, we can derive $H_s^k = V^{* \perp} \cap H^k(\RR^3; \RR)$.
 
When it comes to the invariance of $\calL$, since $P_u$ is commutable with $\calL$ from the definition of $P_u$ through the contour integral (\ref{Pu: definition}), for any $f \in H_u^k=V \cap H^k(\RR^3; \RR)$,
\[
\calL f = \calL P_u f = P_u \calL f \in H_u^k,
\]
hence $H_u^k$ is invariant under $-\calL$. Similarly, $H_s^k$ is invariant under $-\calL$.

\mbox{}

(4) 
From (2) and (3), the operator $-\calL \big|_{H^k_u}$ under the basis $\{ \varphi_j \}_{1 \le j \le N} \subset H^k(\RR^3, \RR)$ is a real matrix denoted by $A$, and its eigenvalue set ${\rm eig}(A) = \Lambda_{\calL} \subset \{ z\in \CC: \Re z > -\frac 12 \delta_g\}$. Hence its real Jordan normal form should be $\text{diag}\{ \tilde J_1, \tilde J_2, ..., \tilde J_l\}$, where $\tilde J_i$ is some blocks given by (\ref{Jordan normal form}) but simply replacing $\frac{\delg}{10}$ by $1$. If we introduce $D_i$ to be the matrix $D_i^{Re} = \text{diag} \{ 1, \delg/10, (\delg/10)^2... \}$ or $D_i^{Im} = \text{diag} \{ I_2, \frac{\delg}{10} I_2, \frac{\delg^2}{10^2} I_2, ... \}$, then
\[
\left(D_i^{Re} \right)^{-1} \tilde J_i^{Re} D_i^{Re} = J_i^{Re}, \text{ and }
\left(D_i^{Im} \right)^{-1} \tilde J_i^{Im} D_i^{Im} = J_i^{Im},
\]
which shows how we obtain (\ref{Jordan normal form}). To show (\ref{increase rate: unstable part}), it suffices to prove each block $\tilde J_i$ satisfies (\ref{Jordan normal form}). In fact, for the block corresponding to real-valued eigenvalue, given any $w \in \RR^{k_i}$ for $k_j= \text{Rank} J_i$, by Cauchy-Schwarz inequality,
\[
w^T \left(J_i^{Re} + \frac{6 \delg }{10} \right) w \ge \sum_{j=1}^{k_i} \frac{\delg}{10} w_i^2 - \frac{\delg}{10} \sum_{j=1}^{k_i-1} |w_j| |w_{j+1}| \ge 0.
\]
For the block corresponding to complex eigenvalue, notice that for any $v \in \RR^2$,
\[
v^T C_i v = [v_1, v_2] \left[ \begin{matrix}
    a_i & -b_i \\
    b_i & a_i
\end{matrix}
\right]
[v_1,v_2]^T
=a_i |v|^2,
\]
For any $w = [w_1^T, w_2^T, ..., w_{k_i}^T]^T \in \RR^{2 k_i}$, with $w_i \in \RR^2$,
we compute
\begin{align*}
    &w^T \left(J_i^{Im} + \frac{6 \delg}{10} I \right)w
    = \sum_{j=1}^{k_i} w_j^T  C_i w_j + \sum_{j=1}^{k_i-1} \frac{\delg}{10} w_j^T w_{j+1} + \frac{6 \delg}{10} w^T w \\
     =& \sum_{j=1}^{k_i} \left( a_i + \frac{6 \delg}{10} \right)|w_j|^2 + \sum_{j=1}^{k_i-1} \frac{\delg}{10} w_j^T w_{j+1} 
     \ge \frac{\delg}{10}\sum_{j=1}^{k_i} w_j^T w_j - \sum_{j=1}^{k_i-1} \frac{\delg}{10} |w_j| |w_{j+1}| \ge 0,
\end{align*}
then (\ref{increase rate: unstable part}) has been verified.
\end{proof}

\begin{lemma}[Smoothing effect of $P_u$]
\label{Lemma: smoothness of V}
For $k \ge 0$, and any $\delg>0$ determined in Proposition \ref{proposition: spectral properties of L}. For every $m \ge k$, the Riesz projection $P_u : H^k(\RR^3; \CC) \to V$ defined by \eqref{Pu: definition} satisfies the bound
\be
\| P_u f \|_{H^{m+1}} \le C \| f \|_{H^m},
\label{Pu: smoothness} 
\ee
for some universal constant $C=C(m)>0$. In particular, $V \subset H^\infty (\RR^3; \CC)$ and $H^k_u \subset H^\infty(\RR^3;\RR)$.
\end{lemma}

\begin{proof}

\noindent \underline{Smoothing of $P_u$ \eqref{Pu: smoothness}.} Recall the  definition of the Riesz projection \eqref{Pu: definition} with contour $\Gamma \subset \{ z \in \CC: \Re z > - \frac \delg 2 \} \cap \rho(-\calL)$. Note that
\[
\calL + \lambda = \calL_0 + \calL' + \lambda = (I + \calL'( \calL_0 +\lambda)^{-1}) (\calL_0 + \lambda).
\]
From Lemma \ref{lemclosmores}, it suffices to show that $I + \calL'( \calL_0 + \lambda)^{-1}: H^m \to H^m$ has an inversion uniformly bounded for $\l \in \Gamma$. 

First of all, we claim that  $\calL'(\calL_0 + \lambda)^{-1}: H^m \to H^m$ is compact. In fact, as shown in Lemma \ref{lemma: compactness of the operator}, we get that $\calL_1', \calL_2': H^{m+1} \to H^m$ is compact, which indicates that $(\calL_1' + \calL_2')( \calL_0 +\lambda)^{-1} : H^m \to H^m$ is compact. Then it suffices to prove that $\calL_3' (\calL_0 + \lambda)^{-1}: H^m \to H^m$ is compact, and it can be verified with the two facts that $(\calL_0+\lambda)^{-1}: H^m \to H^{m+ \frac{3}{2}}$ is bounded and that $\calL_3': H^{m+\frac{3}{2}} \to H^m$ is compact, which has been proved in \eqref{eqresest} and Lemma \ref{lemma: compactness of the operator} respectively. 

Next we show that $I+\calL' (\calL_0+\lambda)^{-1}$ has a bounded inverse on $H^m$ for all $\lambda \in \Gamma \subset \{ z \in \CC: \Re z > - \frac \delg 2 \} \cap \rho(-\calL)$. By contradiction, if we assume that the result is not true, then by Fredholm's alternative, there exists a non-trivial $f \in H^m$ such that
\[
\left( I+\calL' (\calL_0+\lambda)^{-1} \right) f =0, \quad \Rightarrow \quad -\calL f = \lambda f,
\]
a contradiction to $\lambda \in \rho(-\calL)$.

Finally, since the resolvent identity $(\calL_0 + \l)^{-1} - (\calL_0 + \l')^{-1} = (\l' - \l) (\calL_0 + \l)^{-1}(\calL_0 + \l')^{-1}$ implies that $I + \calL' (\calL_0 + \l)^{-1}$ is continuous with respect to $\l$ in the operator norm, the uniformly bounded inversion of $I + \calL' (\calL_0 + \l)^{-1}$ follows the invertibility of each $\l \in \Gamma$ and compactness of $\Gamma$. Hence we have verified \eqref{Pu: smoothness}.

\mbox{}

\noindent \underline{Smoothness of elements in $V$.}
Since $f \in V \subset H^m$, and by the fact that $P_u: H^m \to V$ is a projection, it means that $f = P_u f$. Then by \eqref{Pu: smoothness}, $f \in H^{m+1}$. Iterating this argument again and again, the regularity of $f$ can then be improved to $f \in H^\infty$.
\end{proof}

\mbox{}

\subsection{Modified unstable space}
\label{subsection: modified spectral of L}

In this section, we construct modified unstable space $\tilde H^k_u \subset  C^\infty_0(\RR^3;\RR) \subset L^1(\RR^3;\RR)$ and design the related modified unstable projection $\tilde P_u$, with the linear diagonal dynamics preserved (see Proposition \ref{proposition: spectral property in modified space}).

Besides, from now on, we confine our attention to the real-valued functions and shall represent $H^k$ and $\dot H^{k}$ as $H^k(\RR^3; \RR)$ and $\dot H^{k}(\RR^3; \RR)$ respectively.

\mbox{}

Fix $k \ge 0$, and $\delg>0$ determined in Proposition \ref{proposition: spectral properties of L}. For $\{ \varphi_j, \psi _j\}_{1 \le j \le N}$ given in Proposition \ref{Proposition: unstable mode} (2), we define the matrix by 
\be M_R = \{ (\chi_R \varphi_i, \psi_j)_{H^k} \} _{1 \le i,j \le N},
 \label{matrix}
 \ee
 where $\chi_R$ the cut-off function defined in \eqref{cutoff function: defi}. Noticing that \eqref{eqorthovarphipsi} implies
\bee
M_R = \{  (\varphi_i, \psi_j)_{H^k} + \left( (1-\chi_R) \varphi_i, \psi_j \right)_{H^k} \}_{1 \le i, j \le N} = Id + o_R(1),
\eee
we make the following definition. 

\begin{definition}[Modified unstable space]
\label{def: modified unstable space}
For any $k \ge 0$, and any $\delg>0$ determined in Proposition \ref{proposition: spectral properties of L}, there exists $R = R_{k, \delta_g} \gg 1$ such that the matrix $M_R$ from \eqref{matrix} is invertible, and we define the modified unstable (generalized) eigenmodes $\{\tilde \varphi_i \}_{i=1}^N$ by
\be
(\tilde \varphi_1, ..., \tilde \varphi_N)^T =  M_R^{-1} (\chi_R \varphi_1, \chi_R \varphi_2, ..., \chi_R \varphi_N)^T, 
\label{modified eigenvalues}
\ee
In addition, on the real-valued Sobolev space $H^k$, we define the modified spectral projections as 
\be
\tpu = \sum_{j=1}^N (\cdot, \psi_j)_{H^k} \phirj, \text{ and } \tps= I-\tpu,
\label{modified projection: definition}
\ee
and accordingly, we define the subspaces of $H^k$ by
\be
 \tilde H_u^k = {\rm Ran} (\tpu) = {\rm span}(\tilde \varphi_i)_{i=1}^N,\,\, \text{ and } \tilde H_s^k = {\rm Ran} (\tps).
 \label{modified subspace}
\ee
\end{definition}

Immediately, the following essential properties hold.

\begin{lemma}
\label{lemma: properties of modified projection}
	For any $k \ge 0$, and any $\delg>0$ determined in Proposition \ref{proposition: spectral properties of L}, the projection $\tilde P_u$, $\tilde P_s$ and subspaces $\tilde H^k_u$, $\tilde H^k_s$ from Definition \ref{def: modified unstable space} satisfy the following statements.

	\noindent $\bullet$ $\tpu$ is a projection from $H^k$ to $\tilde H_u^k$, namely $\tilde P_u^2 = \tilde P_u$. Moreover, $H^k = \tilde H_u^k \oplus \tilde H_s^k$.
	
	\noindent  {$\bullet$ $\tilde{H}_u^k \subset C_0^\infty(\RR^3)$.} For any $m \ge k$, $\tpu: H^k \to H^m$ and $\tps: H^m \to H^m$ are bounded. 
	
	\noindent $\bullet$ $\tpu$ and $P_u$ satisfy the following identity:
	\be
	\tilde P_u P_u = \tilde P_u, \quad P_u \tilde P_u = P_u.
	\label{modified projection: property 2}
	\ee
	
	\noindent $\bullet$ $\tps$ and $P_s$ satisfy the following identity:
	\be
	P_s \tilde P_s = \tilde P_s, \quad \tilde P_s  P_s= P_s.
	\label{modified projection: property 3}
	\ee
     Hence $\tilde H_s^k = H_s^k$. 
\end{lemma}

\begin{proof}
    From \eqref{matrix} and \eqref{modified eigenvalues}, we have the orthogonality 
    \be 
(\phir, \psi_j)_{H^k} = \delta_{i,j} \text{ for any } 1 \le i,j \le N.
\label{eqorthoPutilde}
\ee
    which implies the first statement and the decomposition of $H^k$ into direct sum. The choice of $\{\phir\}$ \eqref{modified eigenvalues} and smoothness of $\varphi_i$ from Lemma \ref{Lemma: smoothness of V} also yield $\tilde H^k_u = {\rm span}(\phir)_{i=1}^N = {\rm span}(\chi_R \varphi_i)_{i=1}^N \subset C^\infty_0$. The boundedness of $\tpu$, $\tps$ follows immediately.

    From Proposition \ref{Proposition: unstable mode}, for any $f \in H^k(\RR^3; \RR)$,
	\begin{align*}
		& \qquad \qquad  \tilde P_u P_u f = \sum_{j=1}^N \sum_{i=1}^N (f, \psi_i)_{H^k} (\varphi_i, \psi_j)_{H^k} \phirj = \sum_{j=1}^N (f, \psi_j)_{H^k} \phirj = \tilde P_u f, \\
		& \text{and } \qquad  P_u \tilde P_u f = \sum_{j=1}^N \sum_{i=1}^N (f, \psi_i)_{H^k} (\phir ,\psi_j)_{H^k} \varphi_j = \sum_{j=1}^N (f, \psi_j)_{H^k} \varphi_j=P_u f,
	\end{align*}
	 hence (\ref{modified projection: property 2}) has been verified. And (\ref{modified projection: property 3}) is a simple result of (\ref{modified projection: property 2}). 
  
 Furthermore, for any $f \in \tilde H_s^k$, (\ref{modified projection: property 3}) implies
\[
f = \tps f = P_s \tps f \subset \text{Ran} P_s = H_s^k,
\]
so $\tilde H^k_s \subset H^k_s$. The other side of inclusion follows similarly.
\end{proof}

Next, we discuss the spectral properties of $-\calL$ under this modified spectral decomposition.

\begin{proposition}[\textcolor{black}{Modified spectral properties of $-\calL$}]
\label{proposition: spectral property in modified space}
 For any $k \ge 0$, and $\delg>0$ determined in Proposition \ref{proposition: spectral properties of L}, for $\tpu$, $\tps$ defined in (\ref{modified projection: definition}) and $\tilde H_u^k$, $\tilde H_s^k$ defined in (\ref{modified subspace}), we have the following properties:
 
\noindent $\bullet$ (Modified unstable part) $-\tpu \calL P_u$ maps $\tilde H_u^k$ to $\tilde H_u^k$. Moreover, there is some real basis $\{ \tilde \phi_{i} \}_{1 \le i \le N}$ such that it can be also expressed as $J= \text{diag}\{J_1, J_2, ...,J_l\}$ with $J_i$ given in (\ref{Jordan normal form}). In addition, $J$ satisfies (\ref{increase rate: unstable part}).

\noindent $\bullet$ (Modified stable part) There exists a universal constant $C=C(k)>0$ such that
\be
\| e^{-\tau \calL} f\|_{H^k} \le C e^{-\frac{1}{2} \delg \tau} \| f \|_{H^k}, \quad \forall f \in \tilde H_s^k.
\label{decay rate: stable part modified}
\ee

\end{proposition}
\begin{proof} The second conclusion (\ref{decay rate: stable part modified}) directly follows from (\ref{decay rate: stable part}) and $\tilde H^k_s = H^k_s$ from Lemma \ref{lemma: properties of modified projection}. Now we focus on the modified unstable part. In fact, for any $1 \le i \le N$, we see that 
\[P_u \phir = \sum_{j=1}^N (\phir, \psi_j)_{H^k} \varphi_j = \sum_{j=1}^N \delta_{ij} \varphi_j = \varphi_i,
\]
and 
\[
\tilde P_u \varphi_i = \sum_{j=1}^N  (\varphi_i, \psi_j)_{H^k} \phirj = \sum_{j=1}^N \delta_{ij} \phirj = \phir.
\]
If we denote $A$ as the representation matrix of $-\calL$ under the basis $\{ \varphi_1, \varphi_2, ..., \varphi_N \}$ of $H_u^k$, then according to Proposition \ref{Proposition: unstable mode} and Lemma \ref{lemma: properties of modified projection}, for any $f \in \tilde H_u^k$, expressed as $f=\sum_{i=1}^N f^i \phir$ for some $\bbf = (f^1, f^2, ..., f^N)^T \in \RR^N$,
\begin{align*}
  -\tpu \calL P_u f = -\tpu \calL \sum_{j=1}^N f^j \varphi_j
  = -\tpu 
  \sum_{j=1}^N( A \bbf )^j \varphi_j
  = \sum_{j=1}^N (A \bbf)^j \phirj,
\end{align*}
which implies that the representation matrix of $-\tpu \calL P_u$ under the basis $\{ \phir \}_{1 \le i \le N} \subset \tilde H_u^k$
should also be $A$. According to Proposition \ref{Proposition: unstable mode}, there exists a real-valued transition matrix $Q$ such that $Q^{-1} A Q = J$. Consequently, there exists a real basis $\{ \tilde{\phi}_{i} \}_{1 \le i \le N} \subset \tilde H_u^k$ such that the related representative matrix of $-\tilde P_u \calL P_u$ is $J$.

\end{proof}

\begin{corollary}
\label{corollary: introduce of B norm}
   For any $k \ge 0$, and $\delg>0$ determined in Proposition \ref{proposition: spectral properties of L}, for $\tpu$ defined in (\ref{modified projection: definition}) and $\tilde H_u^k$ defined in (\ref{modified subspace}), there exists an inner product on $\tilde H_u^k$, which we denote by $(\cdot, \cdot)_{\tilde B}$, such that
	\be
(-\tilde P_u \calL P_u f, f)_{\tilde B} \ge - \frac{6 \delta_g}{10} \|f\|_{\tilde B}^2, \quad \forall \; f \in \tilde H_u^k.
\label{growth rate of unstable part: modified part version 2}
\ee
Moreover, $(\cdot, \cdot)_{\tilde B}$ induces a norm $\| \cdot \|_{\tilde B}$ on $\tilde H_u^k$ and $\| \cdot \|_{\tilde B} \simeq \| \cdot \|_{H^k}$.
\end{corollary}

\begin{proof}
	For any $f,g \in \tilde H_u^k$, from Proposition \ref{proposition: spectral property in modified space}, they can be uniquely decomposed into $f= \sum_{i=1}^N \tilde f^i \tilde \phi_{i}$ and $g= \sum_{i=1}^N \tilde g^i \tilde \phi_{i}$ for some $\tilde \bbf= (\tilde f^1, ...\tilde f^N)^T \in \RR^N$ and $\tilde{ \textbf{g}}=(\tilde g^1, ...\tilde g^N)^T \in \RR^N$, then we define
\be
(f,g)_{\tilde B} = \tilde \bbf^\perp \tilde {\textbf{g}}.
\label{inner product on modified unstable space}
\ee
It can be easily verified that \eqref{inner product on modified unstable space} defines an inner product $(\cdot, \cdot)_{\tilde B}$ on $\tilde H_u^k$, inducing a norm $\| \cdot \|_{\tilde B}$ on $\tilde H_u^k$, then \eqref{growth rate of unstable part: modified part version 2} follows from Proposition \ref{proposition: spectral property in modified space} and Proposition \ref{Proposition: unstable mode} (4). Finally, the equivalence $\| \cdot |_{\tilde B} \simeq \| \cdot \|_{H^k}$ is established based on the well-known fact that any two norms on a finite-dimensional space are equivalent.
\end{proof}

\section{Existence of Stable Trajectory}
\label{section: bootstrap}

In this section, we come back to the nonlinear system \eqref{linearized equation} and discuss its dynamics to prove Theorem \ref{main thm}. We fix $k \ge 2$ and $\delta_g < \frac 1{16}$ chosen from Proposition \ref{proposition: spectral properties of L}. Applying the modified spectral decomposition $I = \tps + \tpu$ (see Definition \ref{def: modified unstable space}) to \eqref{linearized equation}, the real-valued variables $(\teps, \tepu,U, \Pi) := (\tilde P_s \ep, \tpu \ep, U, \Pi)$ solve the system
\begin{equation*}
\begin{cases}
	\partial_\tau \tepu = - \tpu\calL \ep + \tpu \left( -U \cdot \nabla \Psi + N(\ep) \right), \\
	\partial_\tau \teps = -\tps \calL \ep + \tps \left( -U \cdot \nabla \Psi + N(\ep) \right), \\
	\partial_\tau U + \frac{1}{2} \left( U + y \cdot \nabla U \right) + U \cdot \na U = \Delta U - \na \Pi - \mu \Psi e_3, \\
	\nabla \cdot U=0,
\end{cases}
\end{equation*}
where 
$\mu(\tau)= \mu_0 e^{-\frac{1}{2} \tau}$
for some $\mu_0 >0$ and $N(\ep)$ is the nonlinear term 
$$ N(\ep) = \na \cdot (\ep \na \Dein \ep). $$ 

In addition, by (\ref{modified projection: property 2}) and the commutability between $\calL$ and $P_u$, we obtain that
\[
-\tilde P_u \calL \ep = -\tilde P_u P_u \calL  \ep = -\tilde P_u \calL P_u \ep = -\tilde P_u \calL  P_u \tilde P_u \ep = -\tilde P_u \calL  P_u \tepu,
\]
and similarly, from \eqref{modified projection: property 3} and the fact that $\tilde H_s^k = H_s^k$, one gets that
\[
-\tps \calL \ep = - \tps \calL \tps \ep - \tps \calL \tpu \ep = -\calL \teps  - \tps \calL \tepu,
\]
hence the system can be rewritten as 
\begin{equation}
\begin{cases}
	\partial_\tau \tepu = - \tpu \calL P_u\tepu + \tpu \left( -U \cdot \nabla \Psi + N(\ep) \right), \\
	\partial_\tau \teps = -\calL \teps - \tps \calL \tepu + \tps \left( -U \cdot \nabla \Psi + N(\ep) \right), \\
	\partial_\tau U + \frac{1}{2} \left( U + y \cdot \nabla U \right) + U \cdot \na U = \Delta U - \na \Pi - \mu \Psi e_3, \\
	\nabla \cdot U=0.
\end{cases}
\label{linearized equation: coupled system of stable, unstable, U term}
\end{equation}

\begin{remark}
    The local wellposedness of \eqref{equation: NS-KS velocity form} given in Theorem \ref{theorem: LWP} implies the local wellposedness of renormalized system \eqref{linearized equation} with $(\e, U) \in H^m \times H^m_\sigma$ for any $m \ge 2$, and thereafter of this spectrally decoupled system \eqref{linearized equation: coupled system of stable, unstable, U term} with $(\tilde \e_s, \tilde \e_u, U) \in \left(\tilde H^k_s \cap H^m\right) \times \tilde H^k_u \times H^m_\sigma$ for $m \ge k \ge 2$. 
\end{remark}

\begin{proposition}[Bootstrap]
\label{proposition: bootstrap}
For any fixed $k \ge 2$, we take $\delta_g < \frac 1{16}$ from Proposition \ref{proposition: spectral properties of L} and let $\tps, \tpu$, $\tilde H_s^k, \tilde H_{u}^k$ and $R = R_{k, \delta_g}$ be given by Definition \ref{def: modified unstable space}. Then there exist $0< \delta_i \ll 1 (0 \le i \le 4)$ with 
\be
\delta_0 \ll \delta_4 \ll \delta_3 \ll \delta_1 \ll \delta_2 \ll \delta_3^{\frac{1}{2}} \ll \min\{\delg, Q(2R)\},
\label{constants: relationship}
\ee
such that for any  {initial datum $\mu(0)=\mu_0$, $\tilde \ep_s(0)=\tilde \ep_{s0} \in \tilde H^k_s \cap H^{k+3}$ and $U(0)=U_0 \in H_{\sigma}^{k+3}$} satisfying
		\be
		|\mu_0| + \| \tilde \ep_{s0} \|_{H^{k+1}} + \textcolor{black}{\| U_0 \|_{\dot H^1 \cap \dot H^{k+1}}}  \le \delta_0,
		\label{bootstrap assum NS-KS: initial data}
		\ee
		there exists $\tilde \ep_{u0} \in \tilde H_u^k$ such that the solution to the (\ref{renormalized equation: NS-KS}) starting from $(\Psi_0, U_0) = (Q+ \tilde \ep_{s0} + \tilde \ep_{u0},U_0)$ globally exists and satisfies the followings for all $\tau \ge 0$:

	\noindent $\bullet$ (Control of the modified stable part $\teps$) 
	\be
	\| \teps (\tau) \|_{H^k} < \delta_1 e^{-\frac{1}{2} \delta_g \tau},
	\label{bootstrap assum: stable part}
	\ee
	
	\noindent $\bullet$ (Control of the higher regularity of the $\teps$)
	\be
	\| \teps(\tau) \|_{\dot H^{k+1}} < \delta_2 e^{-\frac{1}{2} \delg \tau},
	\label{bootstrap assum: stable part higher regularity}
	\ee
	
	\noindent $\bullet$ (Control of the modified unstable part $\tepu$)
	\be
	\| \tepu(\tau) \|_{\tilde B}  \le \delta_3 e^{-\frac{7}{10}\delg \tau},
	 \label{bootstrap assum: unstable part}
	\ee

	\noindent $\bullet$ (Control of the flow)
	\be
	\| U(\tau) \|_{\dot H^1 \cap \dot H^{k+1}} < \delta_4 e^{-\frac{1}{8} \tau}.
	 \label{bootstrap assum: flow}
	\ee
\end{proposition}

Proposition \ref{proposition: bootstrap} is the center of the paper, and it implies Theorem \ref{main thm}. As in \cite{MR3986939,MR4359478}, the Proposition \ref{proposition: bootstrap} will be proven via contradiction using a topological argument as follows: given $(\mu_0, \tilde \ep_{s0}, U_0)$ satisfying \eqref{bootstrap assum NS-KS: initial data} with $\na \cdot U_0 =0$, we assume that for any $\tepu \in \tilde H_u^k$  satisfying \eqref{bootstrap assum: unstable part} when $\tau=0$, the exit time 	
\be
\tau^* = \sup \{ \tau \ge 0:  (\mu,\teps, \tepu, U) \text{  {satisfy (\ref{bootstrap assum: stable part})-(\ref{bootstrap assum: flow}) simultaneously} on } [0,\tau] \},
\label{exiting time}
\ee
\noindent is finite and then we look for a contradiction under the assumptions of the coefficients given in \eqref{constants: relationship}. Subsequently, we therefore study the flow on $[0,\tau^*]$ where \eqref{bootstrap assum: stable part}-\eqref{bootstrap assum: flow} holds. Precisely, we will verify that the bounds in (\ref{bootstrap assum: stable part}), (\ref{bootstrap assum: stable part higher regularity}), and (\ref{bootstrap assum: flow}) can be further improved within the bootstrap regime, ensuring that $(\teps, U)$ will not exit the bootstrap regime at time $\tau=\tau^*$ and the only potential scenario for the solution to exit the bootstrap regime is if \eqref{bootstrap assum: unstable part} fails as $\tau > \tau^*$, thus together with the outgoing flux property of $\tepu$ at the exit time, we can conclude a contradiction via Brouwer's fixed point theorem.

\subsection{A priori estimates}
In this subsection, we consider the a priori estimates. It is worth noting that throughout our discussion, the constant $C>0$ of the estimates may vary from line to line as required.

\subsubsection{Energy estimate of the flow $U$}

\mbox{}

\begin{lemma}[$L^\infty$ smallness of $U$] Under the bootstrap assumptions in Proposition \ref{proposition: bootstrap}, there exists a universal constant $C>0$ such that
\be
 {
\| U(\tau) \|_{\infty} \le C \| U(\tau) \|_{\dot H^1 \cap \dot H^{k+1}} \le C \delta_4 e^{-\frac{1}{8} \tau}, \quad \forall \tau \in [0,\tau^*].
\label{U: L infty}
}
\ee
\end{lemma}

\begin{proof}
With the initial data $\rho_0 \in H^{k+2}$ and $ U_0 \in H_{\sigma}^{k+2}$, the local well-posedness of system \eqref{equation: NS-KS velocity form} implies that $U(\tau) \in H_{\sigma}^{k+2}$ for any $\tau \in [0,\tau^*]$. Consequently, by Sobolev embedding and  {Gagliardo-Nirenberg inequality}, there exists a universal constant $C>0$ such that
\[
\| U \|_{L^\infty} \le C \| U \|_{L^6}^\frac{1}{2} \| U \|_{\dot H^{2}}^\frac{1}{2} \le C \| U \|_{\dot H^1}^\frac{1}{2} \| U \|_{\dot H^{2}}^\frac{1}{2} 
\le C \| U \|_{\dot H^1 \cap \dot H^{k+1}},
\]
then (\ref{U: L infty}) follows (\ref{bootstrap assum: flow}).
\end{proof}

\begin{lemma}[Energy estimates of $U$] 
\label{lemma: improve bootstrap flow}
Under the assumptions of Proposition \ref{proposition: bootstrap}, there exists a universal constant $C>0$ such that
\be
\frac{d}{d\tau} \| U \|_{\dot H^1}^2 \le -\frac{1}{4} \| U \|_{\dot H^1}^2 + C \delta_4^3 e^{-\frac{3}{8} \tau} + C \mu_0^2 e^{-\tau}, \quad \forall \tau \in [0,\tau^*],
\label{energy estimate of U H1 norm}
\ee
and
\be
 \frac{d}{d\tau}  \| U \|_{\dot H^{k+1}}^2
	\le -\frac{1}{4} \| U \|_{\dot H^{k+1}}^2 
	+ C \delta_4^3 e^{-\frac{3}{8} \tau}
	+ C\mu_0^2 e^{- \tau}, \quad \forall \tau \in [0,\tau^*].
	\label{energy estimate of U Hm norm}
\ee
Furthermore, the bootstrap assumption (\ref{bootstrap assum: flow}) can be improved to
\be
\| U (\tau) \|_{\dot H^1 \cap \dot H^{k+1}} \le \frac{1}{2} \delta_4 e^{-\frac{1}{8}  \tau}, \quad \forall \tau \in [0,\tau^*].
\label{improve bootstrap: flow}
\ee
\end{lemma}

\begin{proof} 
\underline{Estimate of \eqref{energy estimate of U H1 norm}.}
By (\ref{linearized equation: coupled system of stable, unstable, U term}), we have 
\begin{align*}
	&\quad \frac{d}{d\tau} \frac{1}{2} \| U \|_{\dot H^1}^2
	= \sum_{|\alpha|=1}\int \partial^\alpha U \cdot \partial^\alpha(\partial_\tau U) dy \\
	& =  \sum_{|\alpha|=1} \int \partial^\alpha U \cdot  \left(  \partial^\alpha \left( - \frac{1}{2} \left( U + y \cdot \nabla U \right) -U \cdot \nabla U + \Delta U - \nabla \Pi - \mu \Psi e_3 \right) \right) dy \\
	& = -\frac{1}{2} \sum_{|\alpha| =1}\int \partial^\alpha U \cdot  \partial^\alpha (U + y \cdot \na U) dy -\sum_{|\alpha| =1} \int \partial^\alpha U \cdot \left( \partial^\alpha (U \cdot \na U) \right) dy  \\ 
 &  \quad -\| U \|_{\dot H^2}^2  
  -\sum_{|\alpha| =1} \mu \int \partial^\alpha U \cdot  \partial^\alpha \Psi e_3 dy,
\end{align*}
where the last equality uses the incompressible condition $\nabla \cdot U=0$. Similar to the argument for proving \eqref{energy estimate: scaling part KS}, we can get that
 {
\be
\frac{1}{2}\sum_{|\alpha| =1}\int \partial^\alpha U \cdot  \partial^\alpha (U + y \cdot \na U) dy  = \frac{1}{4}\| U \|_{\dot H^1}^2.
\label{NS: energy estimate for scaling part}
\ee
}

When it comes to the nonlinear term, by  {Gagliardo-Nirenberg inequality} and Young's inequality, one obtains that
\[
\| f \|_{L^3}^3 \le C\| f \|_{L^2}^\frac{3}{2} \| \nabla f \|_{L^2}^\frac{3}{2} \le \frac{1}{4} \| \nabla f \|_{L^2}^2 + C \| f \|_{L^2}^6, \quad \text{ for any } f \in H^1,
\]
thus it implies that the nonlinear term can be estimated by
\begin{align*}
    &  \quad \Big| \sum_{|\alpha| =1} \int \partial^\alpha U \cdot \left( \partial^\alpha (U \cdot \na U) \right) dy \Big| \\
    &=  \Big| \sum_{|\alpha|=1} \int \partial^\alpha U \cdot \left( ( U \cdot \na  \partial^\alpha U) \right) dy +  \sum_{|\alpha|=1} \int \partial^\alpha U \cdot \left( ( \partial^\alpha U \cdot \na U) \right) dy \Big| \\
    & = \Big|  \sum_{|\alpha|=1} \int \partial^\alpha U \cdot ( \partial^\alpha U \cdot \na U)  dy \Big| \le C\| U \|_{\dot W^{1,3}}^3 \le  \frac{1}{4} \| U \|_{\dot H^2}^2 + C \| U \|_{\dot H^1}^6.
\end{align*}

Additionally, for the buoyancy term, from bootstrap assumption (\ref{bootstrap assum: stable part}) and \eqref{bootstrap assum: unstable part},
\[
\| \mu D \Psi \|_{L^2} = \mu \| DQ + D\teps + D \tepu \|_{L^2} \le C \mu. 
\]

In summary, combining all the estimates above, for any $\tau \in [0,\tau^*]$, by the bootstrap assumptions given in Proposition \ref{proposition: bootstrap}, one obtains that
\begin{align*}
	\frac{d}{d\tau} \frac{1}{2} \| U \|_{\dot H^1}^2
	& \le - \frac{1}{4} \| U \|_{\dot H^1}^2 + C \| U \|_{\dot H^1}^6+ C\mu \| U \|_{\dot H^1}  \\
	& \le  - \frac{1}{8} \| U \|_{\dot H^1}^2+ C \| U \|_{\dot H^1}^6 +  C \mu^2  \le -\frac{1}{8} \| U \|_{\dot H^1}^2 + C \delta_4^6 e^{-\frac{3}{4} \tau} + C \mu_0^2 e^{-\tau},
\end{align*}
which yields (\ref{energy estimate of U H1 norm}).

\noindent \underline{Estimate of \eqref{energy estimate of U Hm norm}.}
Similarly, from \eqref{linearized equation: coupled system of stable, unstable, U term}, one obtains that
\begin{align*}
	& \quad \frac{d}{d\tau} \frac{1}{2} \| U \|_{\dot H^{k+1}}^2 
	= \sum_{|\alpha| =k+1} \int  \partial^\alpha U \cdot \partial^\alpha(\partial_\tau U) dy \\
	& = \sum_{|\alpha| =k+1} \int (\partial^\alpha U) \cdot \left( \partial^\alpha \left( - \frac{1}{2} \left( U + y \cdot \nabla U \right) -U \cdot \nabla U + \Delta U - \nabla \Pi- \mu \Psi e_3 \right) \right) dy \\
	& = - \frac{2k+1}{4} \| U \|_{\dot H^{k+1}}^2 - \| U \|_{\dot H^{k+2}}^2  - \sum_{|\alpha|=k+1}\int \partial^\alpha U \cdot \left( \partial^\alpha (U \cdot \na U) \right) dy 
 - \sum_{|\alpha|=k+1} \mu \int \partial^\alpha U \cdot \partial^\alpha \Psi e_3 dy, 
\end{align*}
here we used $U(\tau) \in H^{k+2}_\sigma$ under the bootstrap assumption by the local wellposedness Theorem \ref{theorem: LWP}.

For the nonlinear term, from  {Gagliardo-Nirenberg inequality}, there exists $C=C(k)>0$ such that
\[
 \| D^j f \|_{L^4} \le C \| Df \|_{L^2}^\frac{k+\frac{5}{4}-j}{k+1}  \| D^{k+2} f \|_{L^2}^{1-\frac{k+\frac{5}{4}-j}{k+1}}, \quad \forall 1 \le j \le k+1,
\]
we have
\begin{align*}
	& \quad \sum_{0 \le i \le k} \Big\| |D^{{k+1}-i} U| \cdot |D^{i+1} U| \Big\|_{L^2}  
 \le \sum_{0 \le i \le {k}}\| D^{{k+1}-i} U \|_{L^4} \| D^{i+1} U \|_{L^4}\\
	&\le C \sum_{0 \le i \le k} \left( \| DU \|_{L^2}^\frac{k+\frac{5}{4}-(k+1-i)}{k+1}  \| D^{k+2} U \|_{L^2}^{1-\frac{k+\frac{5}{4}-(k+1-i)}{k+1}} \right) 
 \left( \| DU \|_{L^2}^\frac{k+\frac{5}{4}-(i+1)}{k+1}  \| D^{k+2} U \|_{L^2}^{1-\frac{k+\frac{5}{4}-(i+1)}{k+1}} \right) \\
 & \le C \| U \|_{\dot H^1}^\frac{k+\frac{1}{2}}{k+1} \| U \|_{\dot H^{k+2}}^{2- \frac{k+\frac{1}{2}}{k+1}},
\end{align*}
then by integration by parts and $\na \cdot U=0$, it implies that the nonlinear term satisfies

\begin{align*}
	& \quad \sum_{|\alpha|=k+1} \int \partial^\alpha U \cdot \left( \partial^\alpha (U \cdot \na U) \right) dy 
       \le C \| U \|_{\dot H^{k+1}} \sum_{0 \le i \le k} \Big\| |D^{{k+1}-i} U| \cdot |D^{i+1} U| \Big\|_{L^2}  \\
       & \le C \|  U \|_{\dot H^{k+1}} \| U \|_{\dot H^{1}}^\frac{k+\frac{1}{2}}{k+1} \|  U \|_{\dot H^{k+2}}^{2- \frac{k+\frac{1}{2}}{k+1}} 
        \le C \| U \|_{\dot H^1 \cap \dot H^{k+1}}^\frac{2k+\frac{3}{2}}{k+1} \|  U \|_{\dot H^{k+2}}^{\frac{k+\frac{3}{2}}{k+1}} 
       \le  \frac{1}{4} \| U \|_{\dot H^{k+2}}^2+  C \| U \|_{\dot H^1 \cap \dot H^{k+1}}^\frac{4k+3}{k+\frac{1}{2}}.
\end{align*}

In addition, for the buoyancy term, by bootstrap assumption (\ref{bootstrap assum: stable part higher regularity}) and \eqref{bootstrap assum: unstable part}, there exists a universal constant $C=C(k)>0$ such that
\bee
\| D^{k+1} \Psi \|_{L^2} = \| D^{k+1} (Q + \teps + \tepu) \|_{L^2} \le C.
\eee

In summary, together with \eqref{U: L infty},
\begin{align}
	\frac{d}{d\tau} \frac{1}{2} \| U \|_{\dot H^{k+1}}^2
	& \le -\frac{2k+1}{4} \| U \|_{\dot H^{k+1}}^2 
	+C \| U \|_{\dot H^1 \cap \dot H^{k+1}}^\frac{4k+3}{k+\frac{1}{2}}  + C \mu \| U \|_{\dot H^{k+1}} 
    \label{U:k+1 regularity}
    \\
	& \le -\frac{1}{8} \| U \|_{\dot H^{k+1}}^2 + C \| U \|_{\dot H^{1} \cap \dot H^{k+1}}^3 + C \mu^2 \notag \\
        & \le -\frac{1}{8} \| U \|_{\dot H^{k+1}}^2 + C \delta_4^3 e^{-\frac{3}{8} \tau} + C \mu_0^2 e^{-\tau}. \notag
	\end{align}
which is \eqref{energy estimate of U Hm norm}.

\mbox{}

\noindent \underline{Estimate of \eqref{improve bootstrap: flow}.}
	Applying Gronwall's inequality onto (\ref{energy estimate of U H1 norm}) and (\ref{energy estimate of U Hm norm}), then together with the conditions given in Proposition \ref{proposition: bootstrap},
	\begin{align*}
		   {\| U(\tau) \|_{\dot H^1 \cap \dot H^{k+1}}^2}
		  &=\| U(\tau) \|_{\dot H^{1}}^2 +\| U(\tau) \|_{\dot H^{k+1}}^2 \\
		 &\le e^{-\frac{1}{4} \tau} \| U(0) \|_{\dot H^1 \cap \dot H^{k+1}}^2 + C\int_0^\tau e^{-\frac{1}{4}(\tau-s)} \left(\delta_4^3 e^{-\frac{3}{8} s} + \mu_0^2 e^{- s} \right) ds  \\
		 &\le C (\delta_0^2 + \delta_4^3 + \mu_0^2) e^{-\frac{1}{4} \tau} \le \frac{1}{4} \delta_4^2 e^{-\frac{1}{4} \tau}.
		\end{align*}
\end{proof}

\subsubsection{Semigroup estimate for $\| \teps \|_{H^k}$}

\begin{lemma}[Semigroup estimate of $\| \teps \|_{H^k}$]
\label{lemma: improve bootstrap stable part}
	Under the assumptions of Proposition \ref{proposition: bootstrap}, the bootstrap assumption (\ref{bootstrap assum: stable part}) can be improved to
	\be
	\| \teps(\tau) \|_{H^k} \le \frac{1}{2} \delta_1 e^{-\frac{\delg}{2} \tau}, \quad \forall \; \tau \in [0,\tau^*].
	\label{improve bootstrap: stable part}
	\ee
\end{lemma}
\begin{proof}
By Duhamel's principle, $\teps$ can be solved as
\begin{align*}
	\teps(\tau) = e^{-\tau \calL} \teps(0) + \int_0^\tau e^{-(\tau-s) \calL}  \left( -\tps \calL \tepu + \tps \left( N(\ep) -U \cdot \nabla \Psi  \right) \right)(s) ds,
\end{align*}
then by Proposition \ref{proposition: spectral property in modified space}, this yields that
\begin{align*}
	 \| \teps(\tau) \|_{H^k}  \le C e^{-\frac{\delg \tau}{2}}  \| \teps(0) \big\|_{H^k} + \int_0^\tau  e^{-\frac{\delg}{2} (\tau-s)}\| -\tps  \calL \tepu + \tps \left(  N(\ep) -U \cdot \nabla \Psi \right) \big\|_{H^k} ds.
\end{align*}

For $ -\tps \calL \tepu$, using the fact that $\calL: H^{k+2} \to H^k$ is bounded, the boundedness of $\tps$ in $H^k$ (Lemma \ref{lemma: properties of modified projection}) and the bootstrap assumption (\ref{bootstrap assum: unstable part}),
\begin{align*}
	\| -\tps \calL \tepu \|_{H^k} \le C \| \calL \tepu \|_{H^k} \le C \| \tepu \|_{H^{k+2}} \le C \| \tepu \|_{\tilde B} \le C \delta_3 e^{-\frac{7 \delg}{10}  \tau},
\end{align*}
for some universal constant $C=C(k)>0$.

For the nonlinear term $\tps N(\ep)$, we apply the bilinear estimate \eqref{energy estimate of nonlinear term} and bootstrap assumptions (\ref{bootstrap assum: stable part}), (\ref{bootstrap assum: stable part higher regularity}), (\ref{bootstrap assum: unstable part}) to find
\begin{align}
    \| \tps N(\ep) \|_{H^k} &\le C \| N(\ep) \|_{H^k} \le C  \| \teps \|_{H^{k+1}}^2 + C\| \tepu \|_{H^{k+1}}^2 \le C \| \teps \|_{H^{k+1}}^2 + C\| \tepu \|_{\tilde B}^2 \notag \\
	& \le C\left( \delta_1 + \delta_2 \right)^2  e^{-\delg \tau} + C \delta_3^2 e^{-\frac{7}{5} \delg  \tau}.
 \label{nonliear estimate: Hk}
\end{align}

When it comes to the advection term,  by  {Gagliardo-Nirenberg inequality}, bootstrap assumptions (\ref{bootstrap assum: stable part}), (\ref{bootstrap assum: stable part higher regularity}) and (\ref{bootstrap assum: unstable part}) together with \eqref{U: L infty}, the $\dot H^{k}$ norm can be controlled as follows:
\begin{align*}
        & \quad \| D^k \left( U \cdot \na \Psi \right) \|_{L^2} \le C \sum_{i=0}^k \Big\| |D^i U| \cdot |D^{k+1-i} \Psi| \Big\|_{L^2} \\
	& \le C \|U \|_{L^\infty} \| D^{k+1} \Psi \|_{L^2} + C \sum_{i=1}^k \| D^i U \|_{L^4} \| D^{k+1-i} \Psi \|_{L^4}  \\
	& \le C \| U \|_{\dot H^1 \cap \dot H^{k+1}} \| D^{k+1} \Psi \|_{L^2} + C \sum_{i=1}^k \| U \|_{\dot H^1 \cap \dot H^{k+1}}  \| \Psi \|_{L^2 \cap \dot W^{k,4}} \\
	& \le C \| U \|_{\dot H^1 \cap \dot H^{k+1}} \| \Psi \|_{H^{k+1}} \le C \| U \|_{\dot H^1 \cap \dot H^{k+1}},
\end{align*}
for some $C=C(k)>0$. A similar argument also implies that 
\[
\| U \cdot \na \Psi \|_{L^2} \le \| U \|_{L^\infty} \| \na \Psi \|_{L^2} \le C \| U \|_{\dot H^1 \cap \dot H^{k+1}} \| \Psi \|_{H^{k+1}} \le C \| U \|_{\dot H^1 \cap \dot H^{k+1}},
\]
so together with Lemma \ref{lemma: properties of modified projection}, we have obtained the $H^k$ estimate of the advection term by
\begin{align}
  \| \tps \left(- U \cdot \nabla \Psi \right) \|_{H^k}
  \le \| U \cdot \na \Psi \|_{H^k} \le C \| U \|_{\dot H^1 \cap \dot H^{k+1}} \le C \delta_4 e^{-\frac{1}{8} \tau}, \quad \forall \tau \in [0,\tau^*].
  \label{advection term: energy estimate}
\end{align}

In summary, with $\delta_g < \frac 1{16}$, we conclude that
\begin{align*}
\| \teps(\tau) \|_{H^k} 
& \le  C e^{-\frac{\delg \tau}{2}}  \| \teps(0) \big\|_{H^k} +  C \int_0^\tau e^{-\frac{\delg}{2}(\tau -s)} \left(  \delta_3 e^{-\frac{7 \delg}{10} s} + \left( \delta_1 + \delta_2 \right)^2  e^{-\delg s} + \delta_3^2 e^{-\frac{7}{5} \delg  s} + \delta_4 e^{-\frac{1}{8} s} \right) ds
\\
& \le C e^{-\frac{\delg \tau}{2}} \left( \delta_0  + (\delta_1 + \delta_2 + \delta_3)^2 + \delta_3 + \delta_4 \right) \le \frac{1}{2} \delta_1 e^{-\frac{\delg \tau}{2}}, \quad \forall \; \tau \in [0,\tau^*].
\end{align*}
\end{proof}

\subsubsection{Energy estimate $ \|\teps \|_{\dot H^{k+1}}$}
\label{subsection: energy estimate}
\begin{lemma}[Energy estimate $ \|\teps \|_{\dot H^{k+1}}$]
\label{lemma: improve bootstrap higher regularity of stable part}
	Under the assumptions of Proposition \ref{proposition: bootstrap}, there exists a constant $C>0$ such that
	\be
	 \frac{d}{d\tau} \| \teps \|_{\dot H^{k+1}}^2
	\le -\| \teps \|_{\dot H^{k+1}}^2 + C\left( \delta_1+ \delta_3 +\delta_4 \right)^2e^{-\delg \tau}, \quad \forall \tau \in [0,\tau^*].
	\label{energy estimate of k+1 derivative}
	\ee
	Furthermore, the constant in (\ref{bootstrap assum: stable part higher regularity}) can be improved to
	\be
	\| \teps(\tau) \|_{\dot H^{k+1}} \le \frac{1}{2} \delta_2 e^{-\frac{1}{2} \delg \tau}, \quad \forall \tau \in [0,\tau^*].
	\label{improve bootstrap: higher regularity of stable part}
	\ee
\end{lemma}
\begin{proof}
\underline{\textit{Step 1. Linear estimates.}}
By integration by parts,  {Gagliardo-Nirenberg inequality} together with (\ref{energy estimate: scaling part KS}), 
\begin{align*}
& (-\calL_0 \teps, \teps)_{\dot H^{k+1}} = (\Delta \teps, \teps)_{\dot H^{k+1}} - \left( \frac{1}{2} \Lambda \teps, \teps\right)_{\dot H^{k+1}}= -\| \teps \|_{\dot H^{k+2}}^2- \frac{2k+3}{4} \| \teps \|_{\dot H^{k+1}}^2, \\
& (\nabla \cdot (\teps \nabla \Dein Q), \teps)_{\dot H^{k+1}}
	= (D^{k+1} \nabla \cdot (\teps \nabla \Dein Q), D^{k+1}\teps)\Ltwo
	= -(D^{k+1} (\teps \na \Dein Q), \na D^{k+1} \teps)_\Ltwo \\
	& \qquad \qquad \qquad \qquad \qquad \quad = -(D^{k+1} \teps \na \Dein  Q, \na D^{k+1} \teps )\Ltwo + O(\| \teps \|_{H^k} \| \teps \|_{\dot H^{k+2}}) \\
	& \qquad \qquad \qquad \qquad \qquad \quad = \frac{1}{2} \int Q |D^{k+1} \teps|^2 dx +  O(\| \teps \|_{H^k} \| \teps \|_{\dot H^{k+2}}) \le C \| \teps \|_{H^k} \| \teps\|_{\dot H^{k+2}}, \\
	&(\nabla \cdot (Q \nabla \Dein \teps), \teps)_{\dot H^{k+1}} = - (D^{k+1}(Q\na \Dein \teps), \na D^{k+1}\teps)_{L^2} \le C\| \teps \|_{H^k} \| \teps \|_{\dot H^{k+2}},
\end{align*}
which, by Young's inequality, implies that
\begin{align*}
	 -(\calL \teps, \teps)_{\dot H^{k+1}} 
	& \le - \| \teps \|_{\dot H^{k+2}}^2 - \frac{2k+3}{4} \| \teps \|_{\dot H^{k+1}}^2 + C \| \teps \|_{H^k} \| \teps \|_{\dot H^{k+2}} \\
	& \le - \frac{7}{8}  \| \teps \|_{\dot H^{k+2}}^2 - \frac{2k+3}{4} \| \teps \|_{\dot H^{k+1}}^2 + C  \| \teps \|_{H^k}^2. 
	\end{align*}
Here we used $\tilde \e_s(\tau) \in H^{k+2}$ under the bootstrap assumption again by the local wellposedness Theorem \ref{theorem: LWP}.

\mbox{}

\noindent \underline{\textit{Step 2. Nonlinear estimates.}}
For $-\tps \calL \tepu$, by Cauchy-Schwarz inequality, the definition of $\tilde B$ norm,
\begin{align*}
	\left( -\tps \calL \tepu, \teps \right)_{\dot H^{k+1}} \le C\| \calL \tepu \|_{H^{k+1}} \| \teps \|_{\dot H^{k+1}} \le C \| \tepu \|_{\tilde B} \|\teps \|_{\dot H^{k+1}} \le \frac{1}{8} \| \teps \|_{\dot H^{k+1}}^2 + C \| \tepu \|_{\tilde B}^2.
\end{align*}

When it comes to the nonlinear term $\tps N(\ep)$, by Sobolev embedding inequality and bilinear estimate \eqref{energy estimate of nonlinear term}, one gets

\begin{align*}
	 \|\tps N(\ep)\|_{\dot H^{k+1}}
	 \le C \| \ep \|_{H^{k+1}}^2 + C\| \ep \|_{H^k} \| \ep \|_{\dot H^{k+2}} 
	 \le C \| \tepu \|_{\tilde B}^2  + C \| \teps \|_{H^k}^2 + C\| \ep \|_{H^k} \| \teps \|_{\dot H^{k+2}},
	\end{align*}
which, by bootstrap assumptions (\ref{bootstrap assum: stable part}), (\ref{bootstrap assum: stable part higher regularity}), (\ref{bootstrap assum: unstable part}), Cauchy-Schwarz inequality and Young's inequality, indicates that
\begin{align*}
	\left( \tps N(\ep), \teps \right)_{\dot H^{k+1}}
	& \le C\left( \| \tepu \|_{\tilde B}^2  + \| \teps \|_{H^k}^2 + \| \ep \|_{H^k} \| \teps \|_{\dot H^{k+2}} \right) \| \teps \|_{\dot H^{k+1}} \\
	& \le  C\left( \| \tepu \|_{\tilde B}^2 + \| \teps \|_{H^k}^2 \right)\| \teps \|_{\dot H^{k+1}} + C \left( \| \teps \|_{H^k} +  \| \tepu \|_{H^k} \right)^{\frac{3}{2}} \| \teps \|_{\dot H^{k+2}}^\frac{3}{2}  \\
	& \le \frac{1}{8} \| \teps \|_{\dot H^{k+1}}^2 + \frac{1}{8} \| \teps \|_{\dot H^{k+2}}^2 + C \left( \| \tepu \|_{\tilde B}^4 + \| \teps \|_{H^k}^4 \right).
\end{align*}

Moreover, as for the advection term, similar to the argument of \eqref{advection term: energy estimate}, under the bootstrap assumptions in Proposition \ref{proposition: bootstrap}, we conclude that 
\[
\| \tps (-U \cdot \na \Psi) \|_{\dot H^{k+1}} \le C 
\| U \cdot \na \Psi \|_{H^{k+1}} \le C  \| U \|_{\dot H^1 \cap \dot H^{k+1}}\left( 1+ \| \teps \|_{\dot H^{k+2}} \right),
\]
which yields that
\begin{align*}
      & \quad \big| (\tps (U \cdot \nabla \Psi), \teps)_{\dot H^{k+1}} \big|
	 \le C \| U \|_{\dot H^1 \cap \dot H^{k+1}} \| \teps\|_{\dot H^{k+1}} \left( 1+ \| \teps \|_{\dot H^{k+2}} \right)  \\
	& \le C\| U \|_{\dot H^1 \cap \dot H^{k+1}} \| \teps\|_{\dot H^{k+1}} + C\| U \|_{\dot H^1 \cap \dot H^{k+1}} \| \teps\|_{H^{k}}^{\frac{1}{2}} \| \teps \|_{\dot H^{k+2}}^\frac{3}{2} \\
	& \le \frac{1}{8} \| \teps \|_{\dot H^{k+1}}^2 +  \frac{1}{8} \| \teps \|_{\dot H^{k+2}}^2 
	+ C  \| U \|_{\dot H^1 \cap \dot H^{k+1}}^2  + C \| \teps \|_{H^k}^2.
	\end{align*}

 \mbox{}

\noindent\underline{\textit{Step 3. Conclusion.}} In summary, from (\ref{linearized equation: coupled system of stable, unstable, U term}), combining all the estimates given above,
\begin{align*}
	\frac{1}{2} \frac{d}{d\tau} \| \teps \|_{\dot H^{k+1}}^2
	&= -(\calL \teps, \teps)_{\dot H^{k+1}} - \left( \tps \calL \tepu, \teps \right)_{\dot H^{k+1}}
	+\left( \tps \left(N(\ep) -U \cdot \na \Psi \right), \teps \right)_{\dot H^{k+1}} \\
	& \le - \frac{7}{8}\| \teps \|_{\dot H^{k+2}}^2 - \frac{2k+3}{4} \| \teps \|_{\dot H^{k+1}}^2 + C  \| \teps \|_{H^k}^2 + \frac{1}{8} \| \teps \|_{\dot H^{k+1}}^2 + C \| \tepu \|_{\tilde B}^2 \\
	& \quad + \frac{1}{8} \| \teps \|_{\dot H^{k+1}}^2 + \frac{1}{8} \| \teps \|_{\dot H^{k+2}}^2 + C \left( \| \tepu \|_{\tilde B}^4 + \| \teps \|_{H^k}^4 \right) \\
	& \quad + \frac{1}{8} \| \teps \|_{\dot H^{k+1}}^2 +  \frac{1}{8} \| \teps \|_{\dot H^{k+2}}^2 
	+ C \| U \|_{\dot H^1 \cap \dot H^{k+1}}^2 + C \| \teps \|_{H^k}^2 \\
	& \le -\frac{k}{2} \| \teps \|_{\dot H^{k+1}}^2 + C \left( \| U \|_{\dot H^1 \cap \dot H^{k+1}}^2+ \| \teps \|_{H^k}^2 + \| \tepu \|_{\tilde B}^2 \right) \\
	& \le -\frac{k}{2} \| \teps \|_{\dot H^{k+1}}^2 + C \left( \delta_1+ \delta_3 +\delta_4 \right)^2e^{-\delg \tau},
\end{align*}
which is (\ref{energy estimate of k+1 derivative}).
{  Here the differentiability of $\tau \mapsto \| \teps \|_{\dot H^{k+1}}^2$ is in the classical sense ensured by Theorem \ref{theorem: LWP} together with the initial regularity assumptions on $(\tilde{\ep}_{s0}, U_0) \in H^{k+3} \times H_{\sigma}^{k+3}$.}
 Hence by Gronwall's inequality,
\begin{align*}
	\| \teps (\tau) \|_{\dot H^{k+1}}^2
	& \le e^{-\tau} \| \teps(0) \|_{\dot H^{k+1}}^2 + C\int_0^\tau e^{-(\tau-s)} \left( \delta_1+ \delta_3 +\delta_4 \right)^2e^{-\delg s} ds \\
	& \le C(\delta_0 + \delta_1 + \delta_3 + \delta_4)^2 e^{-\delg \tau} \le \frac{1}{4} \delta_2^2 e^{-\delg \tau},
\end{align*}
where the last inequality holds from the conditions of $\{\delta_i \}$ given in Proposition \ref{proposition: bootstrap}.
\end{proof}

\subsection{Control of unstable mode}
\begin{proof}[Proof of Proposition \ref{proposition: bootstrap}]
\mbox{}

\noindent \underline{\textit{Improvement of the bootstrap assumptions.}} By contradiction, we assume that there exists $(\mu_0, \tilde \ep_{s0}, U_0)$ satisfying (\ref{bootstrap assum NS-KS: initial data}), such that for any $\tilde \ep_{u0}$ with $\| \tilde \ep_{u0} \|_{\tilde B} \le \delta_3$, the exit time $\tau^*$ defined in \eqref{exiting time} is finite. Then according to Lemma \ref{lemma: improve bootstrap flow}, Lemma \ref{lemma: improve bootstrap stable part}, and Lemma \ref{lemma: improve bootstrap higher regularity of stable part}, combined with the continuity of the solution to system \eqref{linearized equation: coupled system of stable, unstable, U term} with respect to $\tau$, there exists $0 < \tau_\epsilon \ll 1$ such that (\ref{bootstrap assum: stable part}), (\ref{bootstrap assum: stable part higher regularity}), and (\ref{bootstrap assum: flow}) hold for all $\tau \in [0, \tau^* + \tau_\epsilon]$. 
 
 \mbox{}
 
 \noindent \underline{\textit{Outgoing flux property.}} With the argument above, we see that the only scenario for the solution to exit the bootstrap regime is when the bootstrap assumption for the unstable part $\tepu$ given in (\ref{bootstrap assum: unstable part}) does not hold for $\tau > \tau^*$. Besides, from the local wellposedness of \eqref{equation: NS-KS velocity form} given in Theorem \ref{theorem: LWP}, we can see that $\tau \mapsto \tilde \e_u (\tau)$ is continuous. Hence, if we define
  \[
  \tilde\calB(\tau)= \{ v \in \tilde H_u^k: \| v\|_{\tilde B} \le \delta_3 e^{-\frac{7}{10} \delg \tau} \},
  \] 
  then from the previous contradiction assumption, 
   \[
   \begin{cases}
   	\tepu (\tau) \in \tilde\calB (\tau), & \forall \; \tau \in [0,\tau^*], \\
   	\tepu(\tau^*) \in \partial \calB (\tau^*), & \tau = \tau^*.
   \end{cases}
   \]
 Moreover, from (\ref{decay rate: stable part modified}) and (\ref{linearized equation: coupled system of stable, unstable, U term}),
 \begin{align*}
 	\frac{d}{d\tau}\frac{1}{2} \| \tepu \|_{\tilde B}^2
	 & = \left( -\tpu\calL \tpu \tepu, P_u  \right)_{\tilde B} + \left( \tpu \left(-U \cdot \nabla \Psi + N(\ep) \right), \tepu \right)_{\tilde B} \\
	 & \ge - \frac{6 \delg}{10} \| \tepu \|_{\tilde B}^2 + \left( \tpu \left(-U \cdot \nabla \Psi + N(\ep) \right), \tepu \right)_{\tilde B}.
 \end{align*}

Similar to the estimate of \eqref{nonliear estimate: Hk}, by bootstrap assumptions (\ref{bootstrap assum: stable part}), (\ref{bootstrap assum: stable part higher regularity}) and (\ref{bootstrap assum: unstable part}), the nonlinear part can be controlled by
\begin{align*}
	|(\tpu N(\ep), \tepu)_{\tilde B}|
	& \le C \| \tpu N(\ep) \|_{\tilde B} \| \tepu \|_{\tilde B}
	\le C \| N(\ep) \|_{H^k} \| \tepu \|_{\tilde B}  \\
	& \le C \left( \| \teps \|_{H^{k+1}}^2 + \| \tepu \|_{\tilde B}^2 \right) \| \tepu \|_{\tilde B} \\
	& \le C \left( (\delta_1^2 + \delta_2^2) e^{-\delg \tau} + \delta_3^2 e^{-\frac{7}{5} \delg \tau} \right) \delta_3 e^{-\frac{7}{10} \delg \tau} \le C \delta_2^2 \delta_3 e^{-\frac{17}{10}\delg \tau}, \quad 
	\forall \; \tau \in [0,\tau^*],
\end{align*}
for some universal constant $C=C(k)>0$.

For the advection term,  from the estimate obtained in (\ref{advection term: energy estimate}), by Cauchy-Schwarz inequality and bootstrap assumption (\ref{bootstrap assum: unstable part}), one concludes that
\begin{align*}
      |(\tpu \left(U \cdot \na \Psi \right), \tepu)_{\tilde B}|
	 \le C \| U \cdot \nabla \Psi \|_{H^k} \| \tepu \|_{\tilde B} 
	 \le C \delta_3\delta_4 e^{-\left( \frac{1}{8} + \frac{7\delg}{10} \right) \tau}, \quad \forall \tau \in [0,\tau^*].
\end{align*}
In summary, we conclude the outer-going flux property of $\tepu$ as follows: 
\begin{align}
	& \quad \frac{d}{d\tau} \Big|_{\tau =\tau^*} \left( e^{\frac{7}{5} \delg \tau}  \| \tepu \|_{\tilde B}^2 \right) \notag \\
	& \ge \frac{\delg}{5} e^{\frac{7}{5} \delg \tau} \| \tepu \|_{\tilde B}^2  - C \delta_3 \left(\delta_2^2 e^{-\frac{3 \delg}{10} \tau} + \delta_4 e^{-\frac{1}{16} \tau}\right) \Big|_{\tau = \tau^*} \notag\\
	& \ge \frac{\delta_g}{5} \delta_3^2  - C\delta_3(\delta_2^2 + \delta_4)  \ge \frac{\delg}{10} \delta_3^2 >0.
	\label{outer going flux property: NS-KS}
\end{align}

\mbox{}

\noindent \underline{\textit{Brouwer's topological argument.}}
By the local well-posedness theory and \eqref{outer going flux property: NS-KS}, we have the continuity of the mapping $\tilde \ep_{u0} \mapsto \tepu \left(\tau^*(\tilde \ep_{u0}) \right)$, where $\tau^*(\tilde \ep_{u0})$ denotes the exit time with initial data $\tilde \ep_{u0}$. Subsequently, we define a continuous map $\Phi:\overline{B_{\tilde B}(0,1)} \to \partial {B_{\tilde B}(0,1)}$ as follows:
\[
f \in \overline{B_{\tilde B}(0,1)}  \mapsto \tilde \ep_{u0}:= \delta_3 f \in \overline{B_{\tilde B}(0, \delta_3)} \mapsto \tepu(\tau^*) \in \partial \tilde \calB(\tau^*) \mapsto \frac{\tepu(\tau^*)}{\| \tepu(\tau^*) \|_{\tilde B}} \in \partial B_{\tilde B}(0,1),
\]
where $B_{\tilde B}(0,1)$ represents the unit ball centered at the origin in the finite-dimensional space $(\tilde H_u^k, | \cdot |_{\tilde B})$. Particularly, when $f \in \partial B_{\tilde B}(0,1)$, i.e., $\delta_3 f \in \partial B_{\tilde B} (0,\delta_3)$, according to the outer going flux property of the flow as indicated in (\ref{outer going flux property: NS-KS}), $\tau =0$ is the exit time, thus leading to $\Phi(f) = f$. In essence, $\Phi= Id$ on the boundary $\partial B_{\tilde B}(0,1)$. However, upon applying Brouwer's fixed-point theorem to $-\Phi$ on $B_{\tilde B}(0,1)$, we conclude the existence of a fixed point of $-\Phi$ on the boundary $\partial B_{\tilde B}(0,1)$, contradicting the assertion that $\Phi =Id$ on the boundary. Consequently, we have established Proposition \ref{proposition: bootstrap}.
\end{proof}

\subsection{Existence of a finite blowup solution to (\ref{equation: NS-KS velocity form}) with finite mass}

\begin{proof}[Proof of Theorem \ref{main thm}]

\mbox{}

\noindent \underline{\textit{Construction of initial datum.}} Firstly, we choose an arbitrary integer $s \ge 3$, an arbitrary divergence-free vector field $u_0 \in H_\sigma^\infty$. If we set $k=s-1$ and $\delta_g \ll 1$ in Proposition \ref{proposition: bootstrap}, then there exist constants $R \gg 1$ from Definition \ref{def: modified unstable space} and $\{\delta_i\}_{0 \le i \le 4}$ satisfying \eqref{constants: relationship}, such that Proposition \ref{proposition: bootstrap} holds true. 

For the initial velocity field $u_0 \in H_{\sigma}^\infty$ and the constant $0< \delta_0 \ll 1$ determined above, we choose the initial data of scaling parameter $\mu_0 \in (0,\frac{1}{8}\delta_0)$ sufficiently small, such that the initial renormalized velocity field  $U_0 := \mu_0 u_0 \left( \mu_0 \cdot \right)$ satisfies
\[
\| U_0 \|_{\dot H^1 \cap \dot H^s}^2 = \| \mu_0 u_0 (\mu_0 \cdot) \|_{\dot H^1 \cap \dot H^s}^2 = \mu_0 \| u_0 \|_{\dot H^1}^2 + \mu_0^{2s-1} \| u_0 \|_{\dot H^s}^2  \le \left( \frac{1}{8} \delta_0 \right)^2.
\]
In addition, for $Q$ given by (\ref{Q: definition}), we choose $R_0 \gg R \gg 1$ sufficiently large such that 
 \begin{align}
 & \quad \big \| \tps \left( (1-\chi_{R_0}) Q \right) \big\|_{H^{s}} + \big \| \tpu \left( (1-\chi_{R_0}) Q \right) \big\|_{H^{s}}  \notag \\
 &\le C \| (1- \chi_{R_0})  Q \|_{H^{s}} \le \frac{1}{8} \delta_0 \ll Q(2R),
 \label{R0: choice}
 \end{align}
hence Proposition \ref{proposition: bootstrap} holds with the initial datum 
 \be
 (\mu_0, \tilde \ep_{s0}, U_0) = \left(\mu_0, -\tps \left( (1-\chi_{R_0}) Q \right), \mu_0 u_0 \left( \mu_0 \cdot \right) \right),
 \label{initial data}
 \ee
 which yields that there exists $\tilde \ep_{u0} \in \tilde H_u^k$ such that the solution $(\mu, \teps, \tepu,U)$ to (\ref{linearized equation: coupled system of stable, unstable, U term}) globally exists and satisfies (\ref{bootstrap assum: stable part}), (\ref{bootstrap assum: stable part higher regularity}), (\ref{bootstrap assum: unstable part}) and (\ref{bootstrap assum: flow}) on $[0,+\infty)$. In other words, in the self-similar coordinate, the solution $(\mu,\Psi, U)$ to (\ref{renormalized equation: NS-KS}) with initial datum $(\mu_0, \Psi_0, U_0)$, where $\Psi_0$ is defined by
 \be
 \Psi_0 = Q + \tilde \ep_{s0} + \tilde \ep_{u0} = Q - \tps \left( (1-\chi_{R_0}) Q \right) + \tilde \ep_{u0}
 = \chi_{{R_0}} Q  + \tpu \left( (1-\chi_{R_0}) Q \right) + \tilde \ep_{u0},
 \label{initial data}
 \ee
globally exists and the corresponding solution $\Psi(\tau)$ can be decomposed into $\Psi(\tau) = Q +\ep (\tau)$ with
 \begin{align*}
  \| \ep(\tau) \|_{H^{s}}  &  \le \| \teps(\tau) \|_{H^{s}} + \| \tepu (\tau) \|_{H^{s}} \\
  &  \le C \left( \| \teps(\tau) \|_{H^{s}} + \|\tepu(\tau) \|_{\tilde B} \right) \le C (\delta_1 + \delta_2 + \delta_3) e^{-\frac{1}{2}\delg \tau}, \quad \forall \; \tau \ge 0,
 \end{align*}
 for some universal constant $C>0$.

\mbox{}

\noindent \underline{\textit{Non-negativity and finite-mass of the solution.}}
  Observing from (\ref{initial data}) and recalling Lemma \ref{lemma: properties of modified projection} that $\tilde{\ep}_{u0} \in \tilde H_u^{s-1} \subset C^\infty_0$, the initial data $\Psi_0$ constructed in \eqref{initial data} satisfies $\Psi_0 \in C_0^\infty \subset L^1$. Furthermore, according \eqref{constants: relationship}, \eqref{bootstrap assum: unstable part} and \eqref{R0: choice}, by Sobolev embedding inequality, there exists a universal constant $C>0$ such that
\[
\| \tilde P_u \left( \left( 1- \chi_{R_0} \right) Q \right) + \tilde \ep_{u0} \|_{L^\infty} \le
\| \tilde P_u \left( \left( 1- \chi_{R_0} \right) Q \right)\|_{H^{s}} + \|\tilde \ep_{u0} \|_{H^{s-1}} \le C (\delta_0 + \delta_3) \ll Q(2R).
\]
Hence with the previous argument, the radial decreasing of profile $Q$ and the fact that $R_0 \gg R$, within the $2R$-ball $B(0,2R)$, which contains the support of $\tilde P_u \left( \left( 1- \chi_{R_0} \right) Q \right) + \tilde \ep_{u0}$, we obtain that
\begin{align*}
\Psi_0(y) 
&= \chi_{R_0 } Q + \tilde P_u \left( \left( 1- \chi_{R_0} \right) Q \right) + \tilde \ep_{u0} \\
& \ge Q(2R) - \| \tilde P_u \left( \left( 1- \chi_{R_0} \right) Q \right) + \tilde \ep_{u0} \|_{L^\infty} \ge \frac{1}{2} Q(2R) >0, \quad \forall \; |y| \le 2R.
\end{align*}
In addition, outside of $2R$-ball, we have $\Psi_0(y) = \chi_{R_0} Q  \ge 0$ when $|y| \ge 2R$. Consequently, we conclude the non-negativity of the initial density $\Psi_0$.

If we go back to the original system \eqref{equation: NS-KS velocity form}, by \eqref{self-similar coordinate 2} and \eqref{self-similar coordinate},  {then the related initial datum becomes as follows:}
\[
(\rho_0,u_0) = \left( \frac{1}{\mu_0^2} \Psi_0 \left( \frac{x}{\mu_0} \right), u_0 \right),
\]
where $\rho_0 \in C_0^\infty \in L^1$. And by the local-wellposedness theory (see Theorem \ref{theorem: LWP}), this ensures that the corresponding solution $(\rho(t),u(t)) \in H^\infty \times H_\sigma^\infty$ throughout its lifespan.

\noindent \underline{\textit{Finite blowup rate.}} From \eqref{self-similar coordinate 2}, by chain rule,
\[
\frac{d\mu}{dt} = \frac{d \mu}{d\tau} \frac{d \tau}{dt}  = -\frac{\mu}{2} \frac{1}{\mu^2} = -\frac{1}{2 \mu}, \quad \Rightarrow \quad \frac{d}{dt} \mu^2 =-1.
\]
This yields that the solution would blow up at finite time $T = \mu_0^2$ and $\mu$ can be explicitly solved by
\[
\mu(t) = \sqrt{\mu_0^2 -t} = \sqrt{T-t}, \quad \forall \; t \in [0,T).
\]
Hence we have finished the proof of Theorem \ref{main thm}.

\end{proof}

\appendix
\section{Local-wellposedness to Keller-Segel-Navier-Stokes system}
\label{appendix: LWP}
\begin{theorem}[Local wellposedness]
\label{theorem: LWP}
 Fix  $(\rho_0,u_0) \in H^m(\RR^3) \times H_\sigma^m(\RR^3)$ for $m \ge 2$. Then there exist a small constant $T_1 = T_1(\| \rho_0 \|_{H^{\max \{ 2, m-1\}}}, \| u_0 \|_{H^{\max\{ 2,  m-1 \} } })>0$ such that there exists a unique solution 
 { 
 \be
 (\rho,u) \in \bigcap_{j = 0, 1} 
 \left(C^j([0,T_1],H^{m-2j}) \times C^j([0,T_1],H_\sigma^{m-2j}) \right)
 \label{LWP: regularity}
 \ee
 to \eqref{equation: NS-KS velocity form}. Denote the maximal existence time as $T^* = T^*(\rho_0, u_0) \le \infty$. This solution satisfies the following properties:
 \begin{enumerate}
 \item Instant smoothing: \be
 (\rho,u) \in C^\infty([\epsilon, T^*) \times \RR^3) \times C^\infty ([\epsilon,T^*) \times \RR^3), \quad \forall \; \epsilon >0.
 \label{LWP: improve regularity}
 \ee
 \item Blowup criteron: $T^* < \infty$ implies
 \be \lim_{t \to T^*} \left(\|\rho(t)\|_{H^2} + \| u(t)\|_{\dot H^1 \cap \dot H^2} \right)= \infty.
 \label{eqblowupcriterion} \ee
 \item Persistence of non-negativity of density: if $\rho_0$ is nonnegative, then $\rho(t,x)$ stays nonnegative on $[0,T^*)$.
 \end{enumerate}
 }
\end{theorem}

\begin{proof} \underline{\textit{Local wellposedness.}} This is a standard application of semigroup estimate and contraction mapping principle, see for \cite[Chapter 5]{MR4475666}  for the Navier-Stokes case. 

 We write \eqref{equation: NS-KS velocity form} in the Duhamel form
\be
\begin{cases}
  \rho = e^{t\Delta} \rho_0 + \int_0^t e^{(t-s)\Delta} \nabla \cdot \left(-u \rho + \rho \nabla \Delta^{-1} \rho\right) ds, \\
  u = e^{t\Delta} u_0 +   \int_0^t e^{(t-s)\Delta} \mathbb{P} \left( -\nabla \cdot (u \otimes u ) + \rho e_3 \right) ds,
\end{cases}
\label{eqDuhamelNSKS}
\ee
where we use the incompressiblity to rewrite $u \cdot \nabla \rho = \nabla \cdot(u\rho)$ and to apply the Leray projection $\mathbb{P} := I - \nabla \Delta^{-1}\circ  \div$ (see \cite[Section 1.3.5]{MR4475666}) to remove the pressure term. Noticing that for $m \ge 2$, with $m_* := \max \{2, m-1\}$, the aggregation term of Keller-Segel equation can be controlled by
\begin{align} 
\| \rho \nabla \Delta^{-1} \rho \|_{H^m} 
&\le C \| \rho \|_{H^m} \| \nabla \Delta^{-1} \rho\|_{L^\infty}+ C \| \rho \|_{L^\infty} \| \nabla \Delta^{-1} \rho \|_{\dot H^1 \cap \dot H^m} \notag \\
& \quad + C \sum_{\substack{|\a| + |\beta| = m,\\ 1 \le |\a|, |\beta| \le m-1}}\| \pa^\a \rho \|_{L^4} \| \pa^\beta \nabla \Delta^{-1} \rho \|_{L^4} \notag  \\
& \le C \| \rho \|_{H^m} \| \rho \|_{H^{m_*}}, 
\label{energy estimate of nonlinear term}
\end{align}
for some constant $C=C(m)>0$ uniformly in $\rho \in H^m$ with $m \ge 2$. And similarly,
\bee
  \| fg \|_{H^m} \le C \| f \|_{H^m} \| g \|_{\dot H^1 \cap \dot {H^{m_*}} } + C \| g \|_{H^m} \| f\|_{\dot H^1 \cap \dot {H^{m_*}} },
\eee
for any $f,g \in H^m$ with $m \ge 2$.

We apply the boundedness of $\mathbb{P}$ in $H^m$, smoothing estimate of the heat kernel $\| e^{t\Delta}\|_{H^m\to H^{m+1}} \le C t^{-\frac 12}$, together with the boundedness of { Leray projector $\PP$ in $H^m$ (see \cite[Lemma $1.15$]{MR4475666})} to obtain
\bee
  &&\left\|  \int_0^t e^{(t-s)\Delta} \nabla \cdot \left(-u \rho + \rho \nabla \Delta^{-1} \rho\right) ds \right\|_{H^m} \\
  &\le & C \int_0^{t} (t-s)^{-\frac 12} \left( \| u \rho \|_{H^m} + \| \rho \nabla \Delta^{-1} \rho\|_{H^m} \right)(s)  ds \\
  &\le & C t^{\frac 12} \left[  \left( \| \rho \|_{L^\infty_s H^m_x}+ \| u \|_{L^\infty_s H^m_x} \right) \| \rho \|_{L^\infty_s H^{m_*}_x} + \| \rho \|_{L^\infty_s H^m_x} \| u \|_{L^\infty_s(\dot H^1_x \cap \dot H^{m_*}_x)}   \right] \\
  \text{ and } \\
  && \left\|  \int_0^t e^{(t-s)\Delta} \mathbb{P} \left( -\nabla \cdot (u \otimes u ) + \rho e_3 \right) ds \right\|_{H^m} \\
  &\le & C   \int_0^{t}  (t-s)^{-\frac 12}\left( \| u\otimes u\|_{H^{m}} + \| \rho \|_{H^m} \right)(s) ds \\
  &\le & C t^\frac 12  \left(  \| u \|_{L^\infty_s  H^m_x}  \| u \|_{L^\infty_s(\dot H^1_x \cap \dot H^{m_*}_x)} +  \| \rho \|_{L^\infty_s H^m_x} \right).
\eee
{ It is then easy to verify that \eqref{eqDuhamelNSKS} has a unique solution on $[0,T_1]$ by contraction mapping principle for some $T_1>0$ only depending on $ \|u\|_{L^\infty_t (\dot H^1_x \cap \dot H^{m_*}_x)} + \| \rho \|_{L^\infty_t H^{m_*}_x} + 1$. For the regularity \eqref{LWP: regularity}, the $j=0$ case follows the boundedness above and the Duhamel form \eqref{eqDuhamelNSKS}, which further implies the $j=1$ case using the equation \eqref{equation: NS-KS velocity form}. }

{As for \eqref{LWP: improve regularity}, similar to the previous argument, we can easily improve half order of regularity instantly by the smoothing estimate of the heat kernel $\|e^{t\Delta}\|_{H^m \to H^{m+\frac{3}{2}}} \le C t^{-\frac{3}{4}}$, then it is easy to verify \eqref{LWP: improve regularity} by induction.}

In addition, noticing that the lifespan depending on one less order regularity of the solution, the blowup criterion \eqref{eqblowupcriterion} now follows this local wellposedness and induction on $m \ge 2$.

\vspace{0.2cm}

\noindent \underline{\textit{Persistence of non-negativity of the density.}} With $m \ge 2$, this becomes a classical solution to \eqref{equation: NS-KS velocity form}, and the persistence of non-negativity of the density $\rho$ is a standard parabolic maximum principle. We give the proof out of completeness.

{ 
Suppose there exists $0 < T < T^*$ such that $\inf_{(t, x) \in [0, T] \times \RR^3} \rho(t, x) < 0$. Since $\rho_0 \ge 0$ and $\rho(t,x) \in C^0([0, T]; H^2(\RR^3))$ from \eqref{LWP: regularity}, together with the Sobolev embedding $H^2(\RR^3) \subset C^{0,\frac{1}{2}}(\RR^3)$, there exists $(t_*,x_*) \in (0,T] \times \RR^3$ such that
\[ \rho(t_*, x_*) = \min_{(t, x) \in [0, T] \times \RR^3} \rho (t, x) <0. \]
From \eqref{LWP: improve regularity}, this minimum implies 
\[
\pa_t \rho(t_*, x_*) \le 0, \quad \nabla_x \rho (t_*, x_*) = 0, \quad \Delta \rho(t_*, x_*) \ge 0.
\]
Hence we have 
\[
0 = \left[ \partial_t \rho + u \cdot \na \rho -\Delta \rho - \na \rho \cdot \na \Delta \rho - \rho^2 \right]\Big|_{(t_*,x_*)} 
\le -\rho^2(t_*,x_*)<0,
\]
which leads to a contradiction and concludes the proof. 
}
\end{proof}

\section{Behavior of flow}\label{appB}

In this section, we prove the estimates \eqref{growth of u(t)} and \eqref{growth of body pressure nabla pi} regarding singular behavior of fluid part. To begin with, we recall from Proposition \ref{proposition: bootstrap} that
\begin{align}
& \qquad \qquad \; u(t,x) = \frac{1}{\sqrt{T-t}} U\left(t, \frac{x}{\sqrt{T-t}} \right), \quad  \text{ and } \quad 
\rho(t,x) = \bar Q(t,x) + \bar \ep(t,x), \nonumber \\
& \quad \text{ with } \quad \quad  \bar Q(t,x)= \frac{1}{T-t} Q \left( \frac{x}{\sqrt{T-t}} \right),
 \; 
 \quad \text{ and } \quad 
 \bar \ep (t,x) =\frac{1}{T-t} \ep\left( t, \frac{x}{\sqrt{T-t}} \right), \nonumber \\
& \quad \text{ where } \quad \| U(t) \|_{\dot H^1 \cap \dot H^{k+1}} \le \delta_4 (T-t)^{\frac{1}{8}}  \quad \text{ and } \quad  \| \ep(t) \|_{H^{k+1}} \le 2 \delta_2 (T-t)^{\frac{\delta_g}{2}}. \label{eqbddUebar}
\end{align}

\noindent \underline{1. Control of velocity.} 
Via the Duhamel formulation of $u$ \eqref{eqDuhamelNSKS}, we further decompose $u = \sum_{j=1}^4 u^{(j)}$ as
\bee
u^{(1)} = e^{t \Delta} u_0, && u^{(j)} = \int_0^t e^{(t - s)\Delta}  \PP F_j(s) ds \quad {\rm for \,\,}j = 2, 3, 4, \\
{\rm with} \quad&&   F_2 = \bar Q e_3,\quad F_3 = \bar \ep e_3,\quad F_4 = -\nabla \cdot (u \otimes u).
\eee

\mbox{}

\noindent\underline{\textit{Upper bound of $u^{(j)}$.}} We claim for $t \in (0, T)$,
\bea 
\|u^{(j)}(t)\|_{L^\infty} \le \left\{ \begin{array}{ll}
    C \;, & j = 1, 3, 4; \\
    C|\log(T-t)| \;, & j = 2.
\end{array}\right. \label{equpperbdduj}
\eea

Clearly they imply the upper bound in \eqref{growth of u(t)}. 

Indeed, the estimate of $u^{(1)}$ directly follows from $\| e^{t \Delta}\|_{H^2 \to H^2} \le 1 $, $u_0 \in H^2(\RR^3)$ and Sobolev embedding. For $u^{(2)}$ and $u^{(3)}$ we exploit the classical $L^2$-$L^\infty$ estimate
\be
\| e^{t \Delta} \|_{L^2 \to L^\infty} \le C t^{-\frac 34},
\label{heat kernel: decay}
\ee
to integrate the Duhamel formula with $\| \bar Q(t)\|_{L^2} =  (T-t)^{-\frac 14} \| Q \|_{L^2}$ and $\| \bar \e(t)\|_{L^2} \le 2\delta_2(T-t)^{-\frac 14 + \frac{\delta_g}{2}}$ from \eqref{eqbddUebar}. Finally, for $u^{(4)}$, we firstly combine the boundedness of $u^{(j)}$ ($j=1,2,3$), \eqref{heat kernel: decay}, and $\| \nabla u \|_{L^2} = (T-t)^{-\frac 14} \| U \|_{\dot H^1} \le C (T-t)^{-\frac 18}$ from \eqref{eqbddUebar} to compute 
 \begin{align*}
      \| u^{(4)}(t) \|_{L^\infty}
      & \le C \int_0^t (t-s)^{-\frac{3}{4}} \big \| \na \cdot (u \otimes u) \big\|_{L^2}(s) ds 
      \le C \int_0^t (t-s)^{-\frac{3}{4}} \| u \|_{L^\infty} \| \na u \|_{L^2}(s) ds 
      \\
      & \le C \int_0^t (t-s)^{-\frac{3}{4}} \left( \| u^{(4)}(s) \|_{L^\infty} + |\log (T-t)| \right) (T-s)^{-\frac{1}{8}} ds\\
      & \le C + C \int_0^t (t-s)^{-\frac{3}{4}} (T-s)^{-\frac 18} \| u^{(4)}(s) \|_{L^\infty} ds.
 \end{align*}
Note that the boundedness of $u^{(j)}$ for $j = 1, 2, 3$ and \eqref{eqbddUebar} imply that
\[ \| u^{(4)}(t) \|_{L^\infty} \le C|\log (T-t)| + \|u(t) \|_{L^\infty} \le C|\log (T-t)| + \frac{C}{\sqrt{T-t}}\| U \|_{L^\infty} \le C (T-t)^{-\frac 38}.\]
Plugging this boundedness in the integral estimate and iterating 4 times will lead to the constant bound in \eqref{equpperbdduj}.

\mbox{}

\noindent\underline{\textit{Lower bound of $u^{(2)}$.}} Next, we are devoted to showing
\be 
 \| u^{(2)}(t)\|_{L^\infty} \ge C |\log (T-t)|,\quad {\rm as}\,\,t \to T,\label{eqlowerbddu2}
\ee
which implies the lower bound in \eqref{growth of u(t)} in viewing the upper bound of $u^{(j)}$ for $j = 1, 3, 4$ from \eqref{eqbddUebar}. 

Recalling the explicit formula of $Q$ \eqref{Q: definition} and $\PP = I - \nabla \Delta^{-1}\circ \div$, we compute 
 \begin{align}
     &  \Delta^{-1} \pa_3 Q = \frac{4y_3}{2+r^2},\quad \Rightarrow \quad \pa_3 \Delta^{-1} \pa_3 Q = \frac{8+4r^2-8y_3^2}{(2+r^2)^2}, 
     \label{Q: Leray projection result}\\
     \Rightarrow \quad &(\PP (Qe_3))_3 = Q - \pa_3 \Delta^{-1} \pa_3 Q  = \frac{8(2 + y_3^2)}{(2+r^2)^2} \ge \frac{16}{(2+r^2)^2}. \nonumber
\end{align}
Then from the positivity of the integral kernel for $e^{t\Delta}$ and compactness of $[0, T]$, we have 
\[ C_T := \inf_{t \in [0, T]} e^{t\Delta}(\PP (Qe_3))_3 > 0. \]
Consequently, for $t > \frac T2$, 
\[ (u^{(2)})_3(t, 0) = \int_0^t \left[e^{(t-s) \Delta}  (\PP (\bar Q e_3))_3 \right](s,0) ds \ge \int_0^t \frac{C_T}{T-s} ds \ge C_T' |\log (T-t)|, \]
which implies \eqref{eqlowerbddu2}. 

\mbox{}

\noindent\underline{2. Control of pressure}. Since $u$ is divergence free, we take divergence onto both sides of Navier-Stokes equation in \eqref{equation: NS-KS velocity form} to obtain
\[
\div (u \cdot \nabla u) = \Delta \pi - \div(\rho e_3), \quad \Rightarrow \quad 
\na \pi = \na \Dein \div (u \cdot \nabla u) + \na \Dein \div(\rho e_3).
\]
We will prove \eqref{growth of body pressure nabla pi} by showing the following estimates for $t \in (0, T)$
\begin{align}
  \left\| \na \Dein \div (u \cdot \nabla u) \right\|_{L^\infty} 
&\le C (T-t)^{-1+\frac{1}{40}} |\log (T-t)|, \label{eqestpres1} \\
\left\| \na \Dein \div (\bar \e e_3) \right\|_{L^\infty} 
&\le C (T-t)^{-1+\frac{\delta_g}2},  \label{eqestpres2} \\
C^{-1} (T-t)^{-1} \le \left\| \na \Dein \div (\bar Q e_3) \right\|_{L^\infty} 
&\le C (T-t)^{-1}.  \label{eqestpres3}
\end{align}

For \eqref{eqestpres1}, we first apply $\| \na \Dein \div \|_{\dot H^1 \cap \dot H^2 \to L^\infty} \le C$ and  Gagliardo-Nirenberg inequality to see
\begin{align*}
    \Big\| \na \Dein \div (u \cdot \na u) \Big\|_{L^\infty}
    & \le C \| u \cdot \na u \|_{\dot H^1 \cap \dot H^2}
    \le C \| u \|_{L^\infty} \| u \|_{\dot H^1 \cap \dot H^3}.
\end{align*}
Then we recall the energy estimate \eqref{U:k+1 regularity} with $k=2$ reads as
\[
\frac{d}{d\tau} \frac{1}{2} \| U(\tau) \|_{\dot H^{3}}^2
	 \le -\frac{5}{4} \| U(\tau) \|_{\dot H^{3}}^2 
	+C \| U(\tau) \|_{\dot H^1 \cap \dot H^{3}}^\frac{22}{5}  + C \mu_0 e^{-\frac 12 \tau} \| U(\tau) \|_{\dot H^{3}}.
\]
Iterating this with the estimate \eqref{eqbddUebar} yields $\| U(\tau) \|_{\dot H^3} \le C e^{-\frac{11}{40 }\tau} = C (T-t)^{\frac{11}{40}}$. Noticing that $\| u \|_{\dot H^k} = (T-t)^{-\frac{2k-1}{4}} \| U(t)\|_{\dot H^k}$, this conclude \eqref{eqestpres1} by combining with \eqref{eqbddUebar} and \eqref{growth of u(t)}. 

Next, \eqref{eqestpres2} and the upper bound in \eqref{eqestpres3} follow from $\| \na \Dein \div \|_{\dot H^1 \cap \dot H^2 \to L^\infty} \le C$ and the $H^2$ boundedness of $Q$ and $\e$ \eqref{eqbddUebar}. Finally, we exploit the explicit formula \eqref{Q: Leray projection result} to see
\[
(\na \Dein \div (\bar Qe_3))_3(0)
= \frac{1}{T-t}(\partial_3 \Dein \partial_3 Q) (0)=\frac{2}{T-t}. 
\]
That implies the lower bound in \eqref{eqestpres3} and hence we have concluded the proof of \eqref{growth of body pressure nabla pi}.

\mbox{}

\vspace{0.5cm}

\noindent \textbf{Conflict of interest.} The authors declare that there is no conflict of interest regarding the publication of this paper.

\vspace{0.1cm}

\noindent \textbf{Data availability.} No new data were created or analyzed in this study. Data
sharing is not applicable to this article.

\normalem
\bibliographystyle{siam}
\bibliography{Bib-2}

\begin{thebibliography}{10}

\bibitem{MR2768550}
{\sc H.~Bahouri, J.-Y. Chemin, and R.~Danchin}, {\em Fourier analysis and
  nonlinear partial differential equations}, vol.~343 of Grundlehren der
  mathematischen Wissenschaften [Fundamental Principles of Mathematical
  Sciences], Springer, Heidelberg, 2011.

\bibitem{MR3730537}
{\sc J.~Bedrossian and S.~He}, {\em Suppression of blow-up in
  {P}atlak-{K}eller-{S}egel via shear flows}, SIAM J. Math. Anal., 49 (2017),
  pp.~4722--4766.

\bibitem{MR3892424}
\leavevmode\vrule height 2pt depth -1.6pt width 23pt, {\em Erratum:
  {S}uppression of blow-up in {P}atlak-{K}eller-{S}egel via shear flows
  [{MR}3730537]}, SIAM J. Math. Anal., 50 (2018), pp.~6365--6372.

\bibitem{MR4475666}
{\sc J.~Bedrossian and V.~Vicol}, {\em The mathematical analysis of the
  incompressible {E}uler and {N}avier-{S}tokes equations---an introduction},
  vol.~225 of Graduate Studies in Mathematics, American Mathematical Society,
  Providence, RI, [2022] \copyright 2022.

\bibitem{MR4201903}
{\sc P.~Biler}, {\em Singularities of solutions to chemotaxis systems}, vol.~6
  of De Gruyter Series in Mathematics and Life Sciences, De Gruyter, Berlin,
  [2020] \copyright 2020.

\bibitem{MR3411100}
{\sc P.~Biler, I.~Guerra, and G.~Karch}, {\em Large global-in-time solutions of
  the parabolic-parabolic {K}eller-{S}egel system on the plane}, Commun. Pure
  Appl. Anal., 14 (2015), pp.~2117--2126.

\bibitem{Blanchet_Dolbeault_Perthame_globalexistence06}
{\sc A.~Blanchet, J.~Dolbeault, and B.~Perthame}, {\em Two-dimensional
  {K}eller-{S}egel model: optimal critical mass and qualitative properties of
  the solutions}, Electron. J. Differential Equations,  (2006), pp.~No. 44, 32.

\bibitem{Brenner_Constantin_Leo_Schenkel_Venkataramani_steady_state_99}
{\sc M.~P. Brenner, P.~Constantin, L.~P. Kadanoff, A.~Schenkel, and S.~C.
  Venkataramani}, {\em Diffusion, attraction and collapse}, Nonlinearity, 12
  (1999), pp.~1071--1098.

\bibitem{MR2759829}
{\sc H.~Brezis}, {\em Functional analysis, {S}obolev spaces and partial
  differential equations}, Universitext, Springer, New York, 2011.

\bibitem{buckmaster2022smooth}
{\sc T.~Buckmaster, G.~Cao-Labora, and J.~G{\'o}mez-Serrano}, {\em {S}mooth
  imploding solutions for 3{D} compressible fluids}, arXiv preprint
  arXiv:2208.09445,  (2022).

\bibitem{buseghin2023existence}
{\sc F.~Buseghin, J.~Davila, M.~del Pino, and M.~Musso}, {\em {E}xistence of
  finite time blow-up in {K}eller-{S}egel system}, arXiv preprint
  arXiv:2312.01475,  (2023).

\bibitem{MR3932458}
{\sc J.~A. Carrillo, K.~Craig, and Y.~Yao}, {\em Aggregation-diffusion
  equations: dynamics, asymptotics, and singular limits}, in Active particles.
  {V}ol. 2. {A}dvances in theory, models, and applications, Model. Simul. Sci.
  Eng. Technol., Birkh\"{a}user/Springer, Cham, 2019, pp.~65--108.

\bibitem{Chen_Hou_Euler_blowup_theory_2023}
{\sc J.~Chen and T.~Y. Hou}, {\em Stable nearly self-similar blowup of the 2{D}
  {B}oussinesq and 3{D} {E}uler equations with smooth data {I}: {A}nalysis},
  preprint, arXiv:2210.07191,  (2023).

\bibitem{Chen_Hou_Euler_blowup_numerical_2023}
\leavevmode\vrule height 2pt depth -1.6pt width 23pt, {\em Stable nearly
  self-similar blowup of the 2{D} {B}oussinesq and 3{D} {E}uler equations with
  smooth data {II}: {R}igorous {N}umerics}, preprint, arXiv:2305.05660,
  (2023).

\bibitem{Collot_Ghoul_Masmoudi_Nguyen_2DtypeII_blowup22}
{\sc C.~Collot, T.-E. Ghoul, N.~Masmoudi, and V.~T. Nguyen}, {\em Refined
  description and stability for singular solutions of the 2{D} {K}eller-{S}egel
  system}, Comm. Pure Appl. Math., 75 (2022), pp.~1419--1516.

\bibitem{Collot_Ghoul_Masmoudi_Nguyen_3Dblowup_Collasping-ring_blowup23}
\leavevmode\vrule height 2pt depth -1.6pt width 23pt, {\em Collapsing-ring
  blowup solutions for the {K}eller-{S}egel system in three dimensions and
  higher}, J. Funct. Anal., 285 (2023), pp.~Paper No. 110065, 41.

\bibitem{MR3986939}
{\sc C.~Collot, P.~Rapha\"{e}l, and J.~Szeftel}, {\em On the stability of type
  {I} blow up for the energy super critical heat equation}, Mem. Amer. Math.
  Soc., 260 (2019), pp.~v+97.

\bibitem{MR2099126}
{\sc L.~Corrias, B.~Perthame, and H.~Zaag}, {\em Global solutions of some
  chemotaxis and angiogenesis systems in high space dimensions}, Milan J.
  Math., 72 (2004), pp.~1--28.

\bibitem{cui2024suppression}
{\sc S.~Cui, L.~Wang, and W.~Wang}, {\em {S}uppression of blow-up in the 3{D}
  {P}atlak-{K}eller-{S}egel-{N}avier-{S}tokes system via non-parallel shear
  flows}, arXiv preprint arXiv:2401.15982,  (2024).

\bibitem{Dolbeault_Perthame_globalexistence04}
{\sc J.~Dolbeault and B.~Perthame}, {\em Optimal critical mass in the
  two-dimensional {K}eller-{S}egel model in {$\Bbb R^2$}}, C. R. Math. Acad.
  Sci. Paris, 339 (2004), pp.~611--616.

\bibitem{MR2909934}
{\sc R.~Donninger and B.~Sch\"{o}rkhuber}, {\em Stable self-similar blow up for
  energy subcritical wave equations}, Dyn. Partial Differ. Equ., 9 (2012),
  pp.~63--87.

\bibitem{MR3537340}
\leavevmode\vrule height 2pt depth -1.6pt width 23pt, {\em On blowup in
  supercritical wave equations}, Comm. Math. Phys., 346 (2016), pp.~907--943.

\bibitem{MR4334974}
{\sc T.~Elgindi}, {\em Finite-time singularity formation for {$C^{1,\alpha}$}
  solutions to the incompressible {E}uler equations on {$\Bbb R^3$}}, Ann. of
  Math. (2), 194 (2021), pp.~647--727.

\bibitem{MR4445341}
{\sc T.~M. Elgindi, T.-E. Ghoul, and N.~Masmoudi}, {\em On the stability of
  self-similar blow-up for {$C^{1,\alpha}$} solutions to the incompressible
  {E}uler equations on {$\Bbb R^3$}}, Camb. J. Math., 9 (2021), pp.~1035--1075.

\bibitem{engel2000one}
{\sc K.-J. Engel and R.~Nagel}, {\em One-parameter semigroups for linear
  evolution equations}, vol.~194 of Graduate Texts in Mathematics,
  Springer-Verlag, New York, 2000.

\bibitem{Fuest_blowup_optimal_2021}
{\sc M.~Fuest}, {\em Approaching optimality in blow-up results for
  {K}eller-{S}egel systems with logistic-type dampening}, NoDEA Nonlinear
  Differential Equations Appl., 28 (2021), pp.~Paper No. 16, 17.

\bibitem{MR4685953}
{\sc I.~Glogi\'{c} and B.~Sch\"{o}rkhuber}, {\em Stable {S}ingularity
  {F}ormation for the {K}eller--{S}egel {S}ystem in {T}hree {D}imensions},
  Arch. Ration. Mech. Anal., 248 (2024), p.~4.

\bibitem{He_Gong_8pithresholdforNSKS}
{\sc Y.~Gong and S.~He}, {\em On the {$8\pi$}-critical-mass threshold of a
  {P}atlak-{K}eller-{S}egel-{N}avier-{S}tokes system}, SIAM J. Math. Anal., 53
  (2021), pp.~2925--2956.

\bibitem{MR3826109}
{\sc S.~He}, {\em Suppression of blow-up in parabolic-parabolic
  {P}atlak-{K}eller-{S}egel via strictly monotone shear flows}, Nonlinearity,
  31 (2018), pp.~3651--3688.

\bibitem{MR4612142}
\leavevmode\vrule height 2pt depth -1.6pt width 23pt, {\em Enhanced dissipation
  and blow-up suppression in a chemotaxis-fluid system}, SIAM J. Math. Anal.,
  55 (2023), pp.~2615--2643.

\bibitem{helffer2013spectral}
{\sc B.~Helffer}, {\em Spectral theory and its applications}, no.~139,
  Cambridge University Press, 2013.

\bibitem{MR1651769}
{\sc M.~A. Herrero, E.~Medina, and J.~J.~L. Vel\'azquez}, {\em Self-similar
  blow-up for a reaction-diffusion system}, J. Comput. Appl. Math., 97 (1998),
  pp.~99--119.

\bibitem{hillesdon1995development}
{\sc A.~Hillesdon, T.~Pedley, and J.~Kessler}, {\em The development of
  concentration gradients in a suspension of chemotactic bacteria}, Bulletin of
  mathematical biology, 57 (1995), pp.~299--344.

\bibitem{MR2013508}
{\sc D.~Horstmann}, {\em From 1970 until present: the {K}eller-{S}egel model in
  chemotaxis and its consequences. {I}}, Jahresber. Deutsch. Math.-Verein., 105
  (2003), pp.~103--165.

\bibitem{MR2073515}
\leavevmode\vrule height 2pt depth -1.6pt width 23pt, {\em From 1970 until
  present: the {K}eller-{S}egel model in chemotaxis and its consequences.
  {II}}, Jahresber. Deutsch. Math.-Verein., 106 (2004), pp.~51--69.

\bibitem{kevinSashastokesboussi}
{\sc Z.~Hu and A.~Kiselev}, {\em Suppression of chemotactic blowup by strong
  buoyancy in {Stokes}-{Boussinesq} flow with cold boundary}, J. Funct. Anal.,
  287 (2024), p.~46.
\newblock Id/No 110541.

\bibitem{hukiselevyao2023suppression}
{\sc Z.~Hu, A.~Kiselev, and Y.~Yao}, {\em {S}uppression of chemotactic
  singularity by buoyancy}, to appear in Geom. Funct. Anal.,  (2023).
\newblock Available at arXiv:2305.01036.

\bibitem{jiaSverakillposedness}
{\sc H.~Jia and V.~Sverak}, {\em Are the incompressible 3d {N}avier-{S}tokes
  equations locally ill-posed in the natural energy space?}, J. Funct. Anal.,
  268 (2015), pp.~3734--3766.

\bibitem{kim2022self}
{\sc J.~Kim}, {\em Self-similar blow up for energy supercritical semilinear
  wave equation}, arXiv preprint arXiv:2211.13699,  (2022).

\bibitem{Kiselevandxusuppresionof}
{\sc A.~Kiselev and X.~Xu}, {\em Suppression of chemotactic explosion by
  mixing}, Arch. Ration. Mech. Anal., 222 (2016), pp.~1077--1112.

\bibitem{MR3991564}
{\sc H.~Kozono, M.~Miura, and Y.~Sugiyama}, {\em Time global existence and
  finite time blow-up criterion for solutions to the {K}eller-{S}egel system
  coupled with the {N}avier-{S}tokes fluid}, J. Differential Equations, 267
  (2019), pp.~5410--5492.

\bibitem{li2023suppression}
{\sc H.~Li, Z.~Xiang, and X.~Xu}, {\em Suppression of blow-up in
  {P}atlak-{K}eller-{S}egel-{N}avier-{S}tokes system via the {P}oiseuille
  flow}, arXiv preprint arXiv:2312.01069,  (2023).

\bibitem{li2023stability}
{\sc Z.~Li}, {\em On stability of self-similar blowup for mass supercritical
  {NLS}}, arXiv preprint arXiv:2304.02078,  (2023).

\bibitem{MR2901320}
{\sc A.~Lorz}, {\em A coupled {K}eller-{S}egel-{S}tokes model: global existence
  for small initial data and blow-up delay}, Commun. Math. Sci., 10 (2012),
  pp.~555--574.

\bibitem{MR4359478}
{\sc F.~Merle, P.~Rapha\"{e}l, I.~Rodnianski, and J.~Szeftel}, {\em On blow up
  for the energy super critical defocusing nonlinear {S}chr\"{o}dinger
  equations}, Invent. Math., 227 (2022), pp.~247--413.

\bibitem{MR4445443}
\leavevmode\vrule height 2pt depth -1.6pt width 23pt, {\em On the implosion of
  a compressible fluid {II}: {S}ingularity formation}, Ann. of Math. (2), 196
  (2022), pp.~779--889.

\bibitem{MR2729284}
{\sc F.~Merle, P.~Rapha\"{e}l, and J.~Szeftel}, {\em Stable self-similar
  blow-up dynamics for slightly {$L^2$} super-critical {NLS} equations}, Geom.
  Funct. Anal., 20 (2010), pp.~1028--1071.

\bibitem{MR1427848}
{\sc F.~Merle and H.~Zaag}, {\em Stability of the blow-up profile for equations
  of the type {$u_t=\Delta u+|u|^{p-1}u$}}, Duke Math. J., 86 (1997),
  pp.~143--195.

\bibitem{MR1361006}
{\sc T.~Nagai}, {\em Blow-up of radially symmetric solutions to a chemotaxis
  system}, Adv. Math. Sci. Appl., 5 (1995), pp.~581--601.

\bibitem{Naito_Suzuki_typeIIblowup}
{\sc Y.~Naito and T.~Suzuki}, {\em Self-similarity in chemotaxis systems},
  Colloq. Math., 111 (2008), pp.~11--34.

\bibitem{nguyen2023construction}
{\sc V.~Nguyen, N.~Nouaili, and H.~Zaag}, {\em Construction of type {I}-{L}og
  blowup for the {K}eller-{S}egel system in dimensions $3$ and $4$}, arXiv
  preprint arXiv:2309.13932,  (2023).

\bibitem{MR3438649}
{\sc T.~Ogawa and H.~Wakui}, {\em Non-uniform bound and finite time blow up for
  solutions to a drift--diffusion equation in higher dimensions}, Anal. Appl.
  (Singap.), 14 (2016), pp.~145--183.

\bibitem{MR1145013}
{\sc T.~J. Pedley and J.~O. Kessler}, {\em Hydrodynamic phenomena in
  suspensions of swimming microorganisms}, in Annual review of fluid mechanics,
  {V}ol. 24, Annual Reviews, Palo Alto, CA, 1992, pp.~313--358.

\bibitem{Raphael_Schweyer_2DtypeII_blowup14}
{\sc P.~Rapha\"{e}l and R.~Schweyer}, {\em On the stability of critical
  chemotactic aggregation}, Math. Ann., 359 (2014), pp.~267--377.

\bibitem{MR3936129}
{\sc P.~Souplet and M.~Winkler}, {\em Blow-up profiles for the
  parabolic-elliptic {K}eller-{S}egel system in dimensions {$n\geq 3$}}, Comm.
  Math. Phys., 367 (2019), pp.~665--681.

\bibitem{tuval2005bacterial}
{\sc I.~Tuval, L.~Cisneros, C.~Dombrowski, C.~W. Wolgemuth, J.~O. Kessler, and
  R.~E. Goldstein}, {\em Bacterial swimming and oxygen transport near contact
  lines}, Proceedings of the National Academy of Sciences, 102 (2005),
  pp.~2277--2282.

\bibitem{winkler_blowup_passive_flow}
{\sc M.~Winkler}, {\em Can fluid interaction influence the critical mass for
  taxis-driven blow-up in bounded planar domains?}, Acta Appl. Math., 169
  (2020), pp.~577--591.

\bibitem{MR4222377}
{\sc L.~Zeng, Z.~Zhang, and R.~Zi}, {\em Suppression of blow-up in
  {P}atlak-{K}eller-{S}egel-{N}avier-{S}tokes system via the {C}ouette flow},
  J. Funct. Anal., 280 (2021), pp.~Paper No. 108967, 40.

\end{thebibliography}

\end{document}